\documentclass[11pt, a4paper]{article}
\usepackage{amsmath,amssymb,amsthm}
\usepackage{mathtools}
\usepackage{fullpage,color}
\usepackage[colorlinks=true,anchorcolor=blue,filecolor=blue,linkcolor=red,urlcolor=blue,citecolor=blue]{hyperref}     

\usepackage{bm}
\usepackage{booktabs}  
\usepackage{float}   
\usepackage{tikz}
\definecolor{bookcol}{RGB}{168,46,40} 
\definecolor{pagecol}{RGB}{33,90,150}
\usetikzlibrary{shapes.geometric, arrows.meta, positioning, calc} 
\usepackage{geometry} 
\AtBeginDocument{%
  \setlength{\abovedisplayskip}{5pt plus 2pt minus 3pt}%
  \setlength{\belowdisplayskip}{5pt plus 2pt minus 3pt}%
  \setlength{\abovedisplayshortskip}{3pt plus 2pt minus 1pt}%
  \setlength{\belowdisplayshortskip}{5pt plus 2pt minus 2pt}%
}

\newtheorem{theorem}{Theorem}[section]

\newtheorem{lemma}[theorem]{Lemma}
\newtheorem{problem}[theorem]{Problem}
\newtheorem{proposition}[theorem]{Proposition}

\theoremstyle{definition}
\newtheorem{definition}[theorem]{Definition}
\newtheorem{example}[theorem]{Example}

\newtheorem{remark}[theorem]{Remark}
\theoremstyle{plain}

\newcommand{\bk}{\operatorname{bk}}

\newcommand{\T}{\mathcal{T}}

\renewcommand{\leq}{\leqslant}
\renewcommand{\le}{\leqslant}
\renewcommand{\geq}{\geqslant}
\renewcommand{\ge}{\geqslant}

\title{The spectral Erd\H{o}s book theorem: sharp bounds and stability}

\author{Yongtao Li\thanks{Yau Mathematical Sciences Center, Tsinghua University, Beijing, 100084, China. 
Email: \texttt{ytli0921@hnu.edu.cn}.}
\and 
Lele Liu\thanks{School of Mathematical Sciences, Anhui University, Hefei, 230601, China. Email: \texttt{liu@ahu.edu.cn}. Supported by the National
Natural Science Foundation of China (No. 12471320), and Anhui Provincial Natural Science Foundation for Excellent Young Scholars (No. 2408085Y003).}
\and 
Bo Ning\thanks{College of Cryptology and Cyber Science \& College of Computer Science, Nankai University, 
Tianjin, 300350, China. Email: \texttt{bo.ning@nankai.edu.cn}. Partially supported by the National Natural Science
Foundation of China (No. 12371350) and Fundamental Research Funds for the Central Universities, Nankai University (No. 63243151).}
}

\date{}

\begin{document}

\maketitle

\vspace{-0.8cm}
\begin{abstract}
The booksize $\bk(G)$ of a graph $G$ is the largest number of
triangles sharing a common edge. A classical theorem of Edwards,
conjectured by Bollob\'{a}s and Erd\H{o}s, states that every $n$-vertex graph $G$ with
 $e(G) > e(T_{n,2})$ has booksize greater
than $n/6$. Zhai and Lin [J. Graph Theory 102 (2023) 502--520]
asked whether the same conclusion holds under the spectral
condition $\lambda(G)>\lambda(T_{n,2})$, where $\lambda(G)$ is the
spectral radius of the adjacency matrix. We answer this question
in a strong form: every $n$-vertex graph $G\neq T_{n,2}$ with
$\lambda(G)\ge\lambda(T_{n,2})$ satisfies \vspace{-2mm}
\[
\bk(G)\ \ge\ \max\Bigl\{\tfrac13\lambda(G),\ \lambda(G)-\tfrac n3,\
2\lambda(G)-n\Bigr\}. \vspace{-1mm}
\]
Consequently, the condition $\lambda(G)>\lambda(T_{n,2})$ forces
$\bk(G)\ge\lfloor n/6\rfloor+1$. The middle term is a spectral improvement of Edwards' bound $\bk(G)\ge\frac{2m}{n}-\frac n3$, and all three bounds are best possible. 
These results come from the
edge-spectral setting: every graph $G$ with $m$ edges and
$\lambda(G)\ge\sqrt m$ that is not a
complete bipartite graph satisfies \vspace{-2mm}
\[
\bk(G)\ \ge\ \max\Bigl\{\lambda(G)-\tfrac{2m}{3\lambda(G)},\
2\lambda(G)-\tfrac{2m}{\lambda(G)}\Bigr\}, \vspace{-2mm}
\]
which strengthens the bound $\bk(G)\ge\frac13\lambda(G)$ of
Zhao, You, Zeng and Zhang. The first term is attained by infinitely many graphs and the second by all regular Tur\'an graphs. 
As an application of our method, we prove a triangle counting  bound $t(G)\ge\frac13(\lambda(G)+1)(\lambda(G)^2-m)$, which improves the result of Bollob\'{a}s and Nikiforov [J. Combin. Theory Ser B. (2007)]. Finally, we prove stability results 
 at both thresholds: if $\lambda(G)\ge(\frac12-o(1))n$, then either 
$\bk(G)\ge(\frac16-o(1))n$ or $G$ can be made into
$T_{n,2}$ by adding and deleting $o(n^2)$ edges; if
$\lambda(G)\ge(1-o(1))\sqrt m$, then either 
$\bk(G)\ge(\frac13-o(1))\sqrt m$ or $G$ differs from a complete bipartite graph in $o(m)$ edges. 
\end{abstract}




\section{Introduction}\label{sec:introduction}

For a graph $G$, we write $V(G)$ and $E(G)$ for its vertex set and edge set, respectively,
 and set $n=|V(G)|$ and $m=|E(G)|$. 
 The neighborhood and degree of a vertex $v$ are denoted by $N(v)$ and $d(v)$, respectively. The adjacency matrix of $G$ is $A(G)$, and its spectral radius is $\lambda(G)$. 
A book of size $q$, denoted by $B_q$, is a graph obtained from $q$ triangles sharing a common edge.  The booksize $\bk (G)$ is the maximum size of a book contained in $G$; equivalently, 
$$\bk(G) =\max_{uv\in E(G)} |N(u)\cap N(v)|.$$ 

 Mantel's theorem states that every $n$-vertex graph with
more than $\lfloor n^{2}/4\rfloor$ edges contains a
triangle, and that the balanced complete bipartite graph
$T_{n,2}$ is the unique extremal graph. Two classical
strengthenings measure how robust this conclusion is. The
first counts triangles: Rademacher (see Erd\H{o}s
\cite{Erd1955,Erdos1964}) showed that such a graph
contains at least $\lfloor n/2\rfloor$ triangles, which
opened the study of supersaturation; see
\cite{Mub2010,LPS2020,MY2025,LM2022-Erd-Rad,BC2023} for
later developments. The second locates them: Erd\H{o}s
asked how many triangles must share a common edge, and the answer, a book of size $n/6$, is given by the theorem of Edwards recalled below. The booksize is a local quantity, attached to one edge, whereas the number of triangles is a global count; a lower bound on $\bk(G)$ therefore says where the
triangles are, not merely how many there are. Books are a basic device for turning local density into global structure: they drive the upper bound for diagonal Ramsey numbers \cite{Thomason1988,Conlon2009}, they
detect quasirandomness \cite{CFW2022}, and they govern
the trade-off between local and global triangle counts
\cite{Mubayi2012,CFS2020}. For other ways in which the
triangles forced by Mantel's threshold must be
distributed, see \cite{FM2017,GL2018} for edges in
triangles, \cite{Erdos95} for intersecting triangles and
\cite{GK2017} for edge-disjoint triangles.

\smallskip 
The study of books in graphs  has a long history and goes back to the work of Erd\H{o}s \cite{Erdos1962} in 1962, in which 
he proved that there exists a constant $c>0$ such that $m > n^2/4$ forces $\bk(G) > c\,n$. 
Moreover, Erd\H{o}s \cite[page 124]{Erdos1962} announced (without proof) that ``\textit{By more careful considerations we can prove that $\bk (G)\ge n/6+O(1)$}.'' 
In 1969, Erd\H{o}s \cite[Theorem 3]{Erdos1969} revisited this problem and proved that every $n$-vertex graph with $\mathrm{ex}(n,K_p) +1$ edges has an edge which is contained in $(10p)^{-6p}n^{p-2}$ copies of $K_p$. 
Erd\H{o}s \cite[page 291]{Erdos1969} wrote again that ``\textit{In fact I can show that every graph with $\lfloor n^2 /4\rfloor +1$ edges has an edge which is contained in at least $n/6 + O(1)$ triangles and that $n/6$ is best possible. For $p>3$, I have not succeeded in determining the best possible constant.}'' 
Bollob\'as and Nikiforov \cite{BN2005} provided a generalization, and in \cite{BN2008,BN2011} they studied the clique version, improving Erd\H{o}s' bound on the number of  copies of $K_p$ sharing a common edge.

\smallskip  
In 1975, Bollob\'{a}s and Erd\H{o}s  \cite{BollobasErdos1975}  conjectured that the $O(1)$ term in the booksize bound can be removed, that is, every $n$-vertex graph $G$ with $m> \lfloor n^{2}/4\rfloor$ edges satisfies $\bk (G)\ge n/6$. Later, Edwards \cite{Edwards1977} confirmed this conjecture in an unpublished manuscript, as recorded in  \cite[Lemma 4]{EFR1992}, and Khad\v{z}iivanov and Nikiforov \cite{KN1979} proved it independently, as recorded in \cite{BN2005}. Thus,
\begin{equation} 
\label{eq-Edw}
e(G)> e(T_{n,2})
\quad\Longrightarrow\quad
\bk(G)>\frac{n}{6}.
\end{equation}
Both the constant $1/6$ and the strict inequality in the hypothesis are best possible.  
Alternative proofs were given by Bollob\'{a}s and Nikiforov \cite{BN2005} and by Li, Feng and Peng \cite[Sec. 4.3.1]{LiFengPeng2025}. Zhao's textbook \cite{Zhao-book} lists the weaker bound $(1/6-o(1))n$ as Exercise~1.1.10. 
The behavior of graphs near this threshold has also attracted much attention; see, e.g.,  \cite{Mubayi2012,CFS2020, ChenMaWang2026, MiaoLiuDam2026}.

\smallskip
All of the results above are statements about the number
of edges. 
Spectral extremal graph theory asks how large the
spectral radius of the adjacency matrix can be when a graph is forbidden to
contain some structure.  
Over the past two decades, it has become clear
that many Tur\'an-type theorems admit spectral
strengthenings, in which the edge count is replaced by the spectral radius. The prototype is the spectral Tur\'an theorem: Wilf \cite{Wil1986} proved
that every $K_{r+1}$-free graph on $n$ vertices satisfies
$\lambda(G)\le(1-\frac1r)n$, and Nikiforov
\cite{Nikiforov2007} sharpened this to
$\lambda(G)\le\lambda(T_{n,r})$, with equality only for
$T_{n,r}$. Since $\lambda(G)\ge\frac{2m}{n}$, the spectral statement implies the classical Tur\'{a}n theorem, and it is applicable for more graphs.  
This paper investigates the program for the Erd\H{o}s book theorem. 
 
\subsection{The spectral Erd\H{o}s book problem}
 
 \label{sec-vertex-setting}

In 2023, Zhai and Lin \cite{ZhaiLin2023} initiated the corresponding spectral extremal problem for books. 
 They proved that the spectral condition $\lambda(G)\ge \lambda(T_{n,2})$ forces a book of size greater than $2n/13$, unless $G= T_{n,2}$. Furthermore, Zhai and Lin \cite[Problem 1.2]{ZhaiLin2023} proposed the spectral Erd\H{o}s book problem: whether the coefficient $1/6$ remains valid under the spectral hypothesis. 
 
\begin{problem}[Zhai--Lin \cite{ZhaiLin2023}]
\label{prob-ZL}
For every positive integer $n$, is it true that 
\begin{equation*}
\lambda(G)>\lambda(T_{n,2})
\quad\stackrel{?}{\Longrightarrow}\quad
\bk(G)>\frac{n}{6}.
\end{equation*}
\end{problem}

Problem \ref{prob-ZL} asks for a genuine strengthening of \eqref{eq-Edw}. 
Indeed, if $e(G)>e(T_{n,2})$, then
$ \lambda (G) \ge \frac{2m}{n}
 \ge \frac{2}{n}\big( \lfloor \frac{n^2}{4}
\rfloor +1\big)
 > \frac n2 \ge \lambda (T_{n,2})$,
so a positive answer to Problem \ref{prob-ZL} implies
\eqref{eq-Edw}.   
But the converse fails. 
Take $s=\lceil \frac{n}{2} \rceil+2$ and 
$G=K_{s}\cup (n-s)K_{1}$. Then
$\lambda(G) = s-1 > \frac{n}{2} \ge \lambda(T_{n,2})$ while
$e(G) = \binom{s}{2}\approx \frac{1}{8}n^{2}$ is only half of the Mantel threshold $\frac{1}{4}n^2$. 
The spectral radius reacts to local density, whereas the edge count sees the global average. This is precisely what makes the spectral form of
the Erd\H{o}s book theorem a statement about a much larger family of graphs than \eqref{eq-Edw}. 

\smallskip 
In this paper, we solve Problem \ref{prob-ZL}  
of Zhai and Lin \cite{ZhaiLin2023} in a stronger form.

\begin{theorem}  \label{thm-confirm-Zhai-Lin}
Let $G$ be an $n$-vertex graph with 
$\lambda (G) \ge \lambda(T_{n,2})$ and $G\neq T_{n,2}$. 
Then
\begin{equation*}
\bk(G) \ \ge \  \frac{1}{3}\lambda (G).
\end{equation*}
Consequently, if $G$ is an $n$-vertex graph with 
$ \lambda(G)>\lambda(T_{n,2})$,
then 
$\bk(G)\ge \left\lfloor\frac n6\right\rfloor+1$.
\end{theorem}

Both the hypothesis and the conclusion of Theorem \ref{thm-confirm-Zhai-Lin} are best possible. 
In Section \ref{sec:sharpness}, we characterize the graphs attaining equality: they are exactly the $\frac n2$-regular graphs in which every edge lies in $0$ or $\frac n6$ triangles (Proposition \ref{prop:equality}). 
There are infinitely many extremal graphs of Theorem \ref{thm-confirm-Zhai-Lin}. 
For every integer $n$ with  $12\mid n$, there are at least $\frac n{24}+2$ extremal graphs (Proposition \ref{prop:many}). 
Hence, the bound $\frac13\lambda(G)$ is attained, and the strict inequality in the second statement cannot be relaxed; 
the bound $\lfloor n/6\rfloor+1$ is attained
by unbalanced blow-ups of these graphs (Proposition \ref{prop:triangle-lift}).

\subsection{Refining the Edwards bound}

Edwards \cite{Edwards1977} in fact proved a stronger result than
\eqref{eq-Edw}. As recorded in the papers  \cite{EFR1992,BN2005}, every
$n$-vertex graph $G$ with $e(G) > e(T_{n,2})$ satisfies
\begin{equation}\label{eq:edw-2}
\bk(G) \ \ge \  \frac{2m}{n}-\frac n3 ,
\end{equation}
and \eqref{eq:edw-2} is best possible; see \cite{BN2005}. Both \eqref{eq-Edw} and \eqref{eq:edw-2} follow from a single
inequality, which appears as Theorem 1 in Khad\v{z}iivanov
\cite{Khadziivanov1991} and as equation (8) in Bollob\'as and
Nikiforov \cite{BN2005}, 
\begin{equation} \label{eq-BN}
\big(6\bk (G) - n\big)\, t(G)\ \ge\ \bk (G)\,(Q - mn),
\end{equation}
where $t(G)$ denotes the number of triangles of $G$ and
$Q=\sum_{v\in V(G)}d^2(v)$. We give a short new proof of \eqref{eq-BN} in Appendix
\ref{sec:App}, and we establish its spectral counterpart 
\eqref{eq:spectral-BN} in Section \ref{sec:weighted}.

The quantity $2m/n$ in \eqref{eq:edw-2} is the average degree, and $\lambda (G)\ge \frac{2m}{n}$ always holds. 
It is natural to ask whether the average degree can be replaced by the spectral radius. Our next theorem
answers this question and adds a further bound that takes over in the dense range. 

\begin{theorem}\label{thm:fixed-order-intro}
Let $G$ be an $n$-vertex graph with 
$\lambda (G) \ge \lambda(T_{n,2})$ and $G\neq T_{n,2}$. 
Then
\begin{equation*}
\bk(G) \ \ge \  \max\left\{ \lambda (G)-\frac{n}{3} , \, 2\lambda (G) - n \right\}.
\end{equation*}
\end{theorem}

The first bound $\bk (G)\ge \lambda (G)-\frac n3$ refines the aforementioned bound \eqref{eq:edw-2} due to Edwards. The second bound $\bk(G)\ge 2\lambda-n$ is the spectral counterpart of the inequality $\bk(G)\ge\frac{4m}{n}-n$ of Khad\v{z}iivanov \cite[Lemma 2]{Khadziivanov1991}.  
The second bound is better once
$\lambda (G)>\frac{2n}{3}$, and it is what makes the bound tight on the Tur\'an graphs $T_{n,r}$ with $r\ge 3$ and $r\mid n$. The graphs attaining equality in either bound are determined
in Proposition \ref{prop:equality-fixed}. 
We derive both bounds from the stronger edge-spectral versions (see Theorems \ref{thm:edge-spectral-intro} and \ref{thm:two-lambda}), to which we turn in the next subsection.

\subsection{Books in graphs with given size}

In the second branch of spectral extremal graph theory, the number of edges rather than the number of vertices is prescribed. 
This is called the \emph{edge-spectral Tur\'{a}n problem}, which asks for the maximum spectral radius of an $F$-free graph with $m$ edges; see, e.g., \cite{LLZ-edge-spectral,LLZ-edge-color-critical,LLLZ2026}. 
It goes back to Brualdi and Hoffman
\cite{BrualdiHoffman1985},
and it differs from the vertex-spectral  setting of \cite{ZhaiLin2023}. 
When the number of edges is the prescribed parameter, 
isolated vertices are irrelevant: they change neither $\lambda (G)$ nor $\bk (G)$.
 \emph{In the edge-spectral setting, we assume that $G$ contains no isolated vertex, and extremal graphs are described accordingly}.

Nosal's theorem \cite{Nosal1970} states that every triangle-free graph with $m$ edges satisfies $\lambda(G)\leq\sqrt m$; see \cite{LNW2021,ZS2022dm,LLZ2025eujc} for refinements. In 2002, 
 Nikiforov \cite{Nikiforov2002} showed that every $K_{r+1}$-free graph $G$  satisfies 
 $\lambda^2(G) \le (1-\frac{1}{r} )2m. $ 
  We refer to \cite{CY2026, LiuNing2026jctb,LLZ-edge-spectral, LLZ-edge-color-critical,FLZ2026} for further extensions.   
Graphs
with $\lambda(G)>\sqrt m$ are called \emph{Nosal graphs}
\cite{LiLiuZhang2026}. They are forced to contain many
triangles:  at least $\frac{1}{3}\lambda (\lambda^2 - m )$ by Bollob\'as and Nikiforov \cite{BN2007}, 
at least $\lfloor\frac12(\sqrt m-1)\rfloor$ by
Ning and Zhai \cite{NingZhai2023}, and at least
$m(\lambda -\sqrt m\,)$ by Chen, Li and Tang
\cite{ChenLiTang2026}.  
Zhai, Lin and Shu \cite{ZLS2021} showed that every Nosal graph contains a copy of $K_{2,t}$ with $t>\frac{1}{4}\sqrt m$, and Li, Liu and Zhang \cite{LLZ-edge-spectral} improved this to $t\ge \frac{1}{2}\sqrt{m}-O(1)$.

How large a book must a Nosal graph contain? 
Every triangle contributes to the codegree of each of its three edges,  so $\bk(G)\ge 3t(G)/m$, and the triangle bound of
\cite{ChenLiTang2026} yields $\bk(G)\ge3(\lambda -\sqrt m\,)$. Consequently, if $\lambda (G)\ge (1+ \varepsilon )\sqrt{m}$ for some $\varepsilon >0$, then $\bk (G)\ge 3\varepsilon \sqrt{m}$. 
The constant $3\varepsilon$ vanishes as $\varepsilon \to 0$,  and it is  interesting to determine the correct  order of the booksize of Nosal graphs. 
Using a counting argument, Nikiforov \cite{Nikiforov2021} first proved $\bk(G)>\frac{1}{12}m^{1/4}$, 
which settles a conjecture of Zhai, Lin and Shu
\cite[Conjecture 5.2]{ZLS2021}.
 Using the probabilistic method, 
 Li, Liu and Zhang \cite{LiLiuZhang2026} obtained
$\bk (G)>\frac{1}{24}\sqrt m$, which gives the correct order of magnitude, 
and they provided a construction showing that the 
constant cannot exceed $\frac13$. This led to the following problem.

\begin{problem}[Li--Liu--Zhang \cite{LiLiuZhang2026}]
\label{prob-LLZ}
Does every $m$-edge graph $G$ satisfy
\begin{equation} \label{eq-ZYZZ}
\lambda(G) > \sqrt{m}
\quad\stackrel{?}{\Longrightarrow}\quad
\bk(G)>\frac{1}{3}\sqrt{m}. 
\end{equation}
\end{problem}

The constant was improved to $\frac19$ by Zhai, Li and Lou
\cite{ZhaiLiLou2026} through Perron vector analysis, and
then to $\frac14$ by Chen, Li and Tang \cite{ChenLiTang2026}. 
Zhao, You, Zeng and Zhang \cite{ZYZZ2026} 
confirmed Problem \ref{prob-LLZ} and 
reached the optimal constant in the stronger form
$\bk (G)\ge \frac13\lambda (G)$, valid for every $m$-edge
graph with $\lambda (G)\ge \sqrt m$ that is not a complete bipartite graph.
The bound $\frac13\lambda (G)$ discards the edge count
entirely. Our next theorem keeps it and is sharp on a whole family of graphs. 

\begin{theorem} \label{thm:edge-spectral-intro}
Let $G$ be a graph with $m\ge 1$ edges and
spectral radius $\lambda \ge\sqrt m$. Then
\[
\bk(G) \  \ge \  \lambda-\frac{2m}{3\lambda},
\]
unless $G$ is a complete bipartite graph. 
\end{theorem}

Theorem \ref{thm:edge-spectral-intro}  improves the bound of Zhao et al. \cite{ZYZZ2026}, since 
$\lambda -\frac{2m}{3\lambda }\ge \frac{1}{3}\lambda$ under the hypothesis $\lambda \ge \sqrt m$. 
The gain grows with $\lambda$: writing $x= {\lambda }/{\sqrt m}\in
[1,\sqrt 2)$, we have
$(\lambda -\frac{2m}{3\lambda } ) /
(\frac{1}{3}\lambda )=3-\frac{2}{x^2} \to 2$ as $x \to \sqrt{2}$. 
Theorem \ref{thm:edge-spectral-intro} implies that 
 if $\lambda (G) \ge \sqrt{m}$ and $G$ is not complete bipartite, then 
$ \bk (G) \ge  \lambda (G) - \frac{2}{3}\sqrt{m} $. 
This is an edge-spectral counterpart of the first term of 
Theorem \ref{thm:fixed-order-intro}.

\smallskip
The bound of Theorem \ref{thm:edge-spectral-intro} is best possible and attained by infinitely many values of $x:=\lambda/\sqrt m$.
The graphs $H_{s,t}$ of Example \ref{ex:EFR} satisfy
$x=2\sqrt{(s-1)/(3s)}$, and these values increase from $x=1$
to the limit point $2/\sqrt3$; the Tur\'an graph $T_{3k,3}$
sits at the limit point, with $\lambda=2k$, $m=3k^2$ and
$\bk=k$. Proposition \ref{prop:equality-edge} characterizes
all graphs attaining equality above the Nosal threshold. Besides the graphs $H_{s,t}$,
they include the joins $F\vee\overline{K}_r$ of an $r$-regular
triangle-free graph $F$ on at most $6r$ vertices with the
empty graph $\overline{K}_r$ (Remark \ref{rem:equality-edge});
the Tur\'an graph $T_{3k,3}$ is the case $F=K_{k,k}$.

\smallskip 
The second bound needs no spectral hypothesis, and it is a companion of Theorem \ref{thm:edge-spectral-intro}.

\begin{theorem}\label{thm:two-lambda}
Let $G$ be a graph with $m\ge1$ edges and spectral
radius $\lambda$. Then
\begin{equation}\label{eq:two-lambda}
\bk(G)\ \ge\ 2\lambda-\frac{2m}{\lambda}.
\end{equation}
Equality holds if and only if $G$ is complete bipartite, or regular complete multipartite.
\end{theorem}

Table \ref{tab:three-settings} compares the bounds on the booksize in the three different settings. 

\begin{table}[H]
\centering
\renewcommand{\arraystretch}{1.3}
\setlength{\tabcolsep}{8pt}
\begin{tabular}{@{}ccc@{}}
\toprule
 Classical edge setting & Vertex-spectral setting
& Edge-spectral setting \\[-4pt]
 $e(G)>e(T_{n,2})$ & $\lambda (G)>\lambda(T_{n,2})$
& $\lambda (G)>\sqrt m$ \\
\midrule
 $\bk (G)>\frac{1}{6}n$
  {\footnotesize\cite{Edwards1977,KN1979}}
& $\bk (G)> \frac{1}{6}n $
& $\bk (G)>\frac{1}{3} \sqrt m$
  {\footnotesize\cite{ZYZZ2026}} \\
    & $\bk (G) \ge\frac13\lambda$
& $\bk (G) \, \ge\, \frac13\lambda$ 
  {\footnotesize\cite{ZYZZ2026}} \\
 $\bk (G) \ge\frac{2m}{n}-\frac n3$
  {\footnotesize\cite{Edwards1977}}
& $\bk (G) \ge\lambda-\frac n3$
& $\bk (G) \ge\lambda-\frac{2m}{3\lambda}$ \\
 $\bk (G) \ge\frac{4m}{n}-n$
  {\footnotesize\cite{Khadziivanov1991}}
& $\bk (G) \ge2\lambda-n$ 
& $\bk (G) \, \ge\, 2\lambda-\frac{2m}{\lambda}$  \\
\bottomrule
\end{tabular}
\caption{Lower bounds on the booksize in the three settings. The bounds in the entries without a reference are established in the current paper.}
\label{tab:three-settings}
\end{table}

\newpage 
The bounds
$\lambda -\frac{2m}{3\lambda }$ and
$2\lambda -\frac{2m}{\lambda }$ are equal when
$\lambda/ \sqrt{m} = 2/ \sqrt{3}$; 
see Figure \ref{fig:two-bounds}. 
 The first is larger below this value, and the second is larger above it. Together they
cover the full range $x=\lambda /\sqrt m\in [1,\sqrt 2\,)$, where the upper limit $\sqrt 2$ is approached by complete graphs. 
By $\lambda \ge \frac{2m}{n}$, we get 
$\lambda -\frac{2m}{3\lambda }\ge \lambda -\frac n3$ and
$2\lambda -\frac{2m}{\lambda }\ge 2\lambda -n$, so Theorems
\ref{thm:edge-spectral-intro} and \ref{thm:two-lambda} are
stronger than Theorem \ref{thm:fixed-order-intro}.

\begin{figure}[htbp]
\centering
\begin{tikzpicture}[xscale=20,yscale=3,
  >={Stealth[length=2mm]},
  dot/.style={circle,fill,inner sep=1.1pt},
  odot/.style={circle,draw,fill=white,inner sep=1pt}]

\fill[black!7]
  (1,0)
  -- plot[domain=1:1.15470,samples=60] (\x,{\x-2/(3*\x)})
  -- plot[domain=1.15470:1.41421,samples=60]
       (\x,{2*\x-2/\x})
  -- (1.41421,0) -- cycle;

\draw[black!30,dashed] (1,0) -- (1,0.33333);
\draw[black!30,dashed] (1.15470,0) -- (1.15470,0.57735);
\draw[black!30,dashed] (1.41421,0) -- (1.41421,1.41421);
\draw[black!30,dashed] (0.988,0.33333) -- (1,0.33333);
\draw[black!30,dashed] (0.988,0.57735) -- (1.15470,0.57735);
\draw[black!30,dashed] (0.988,1.41421) -- (1.41421,1.41421);

\draw[->] (0.982,0) -- (1.5,0);
\draw[->] (0.988,-0.03) -- (0.988,1.60);
\node[anchor=west,font=\small] at (1.415,0.1)
  {$x=\lambda/\sqrt m$};
\node[anchor=south west,font=\small] at (0.99,1.45)
  {$y= \bk(G)/\sqrt m$};

\foreach \p/\l in {1/{$1$},
  1.15470/{$\tfrac{2}{\!\!\sqrt3}$}, 1.41421/{$\sqrt2$}}
  \draw (\p,0) -- (\p,-0.03) node[below,font=\small] {\l};
\foreach \q/\l in {0.33333/{$\tfrac13$},
  0.57735/{$\tfrac{1}{\!\!\sqrt3}$}, 1.41421/{$\sqrt2$}}
  \draw (0.988,\q) -- (0.980,\q) node[left,font=\small] {\l};

\draw[red!75!black,line width=1pt]
  plot[domain=1:1.15470,samples=80] (\x,{\x-2/(3*\x)});
\draw[red!75!black,thin,dashed,opacity=0.5]
  plot[domain=1.15470:1.41421,samples=80]
    (\x,{\x-2/(3*\x)});
\draw[blue!70!black,line width=1pt]
  plot[domain=1.15470:1.41421,samples=80] (\x,{2*\x-2/\x});
\draw[blue!70!black,thin,dashed,opacity=0.5]
  plot[domain=1:1.15470,samples=80] (\x,{2*\x-2/\x});

\draw[green!55!black,densely dotted,line width=0.9pt]
  (1,0.33333) -- (1.41421,0.47140);
\node[green!55!black,font=\scriptsize,anchor=west]
  at (1.254,0.53) {$y\ge \tfrac{1}{3}x$ Zhao et al. \cite{ZYZZ2026}};
  

\foreach \p/\q in {1.15470/0.57735}
  \node[dot,blue!70!black] at (\p,\q) {};
\node[odot,blue!70!black] at (1.41421,1.41421) {};
\node[odot] at (1,0) {};

\node[red!75!black,font=\small,align=center,anchor=west]
  at (1.005,0.82) {Theorem \ref*{thm:edge-spectral-intro}\\
  $y \ge x-\tfrac{2}{3x}$};
\draw[red!75!black,thin,->] (1.085,0.655) -- (1.09,0.5);
\node[blue!70!black,font=\small,align=center,anchor=west]
  at (1.18,1.2) {Theorem \ref*{thm:two-lambda}\\
  $y \ge 2x-\tfrac{2}{x}$};
\draw[blue!70!black,thin,->] (1.235,1.06) -- (1.245,0.91);
\node[black!55,font=\scriptsize] at (1.20,0.18)
  {no graph lies in the gray region};
\node[red!75!black,font=\scriptsize,anchor=south west]
  at (0.999,0.42) {$H_{s,t} (s\ge 4)$};
\node[blue!70!black,font=\scriptsize,anchor=south east]
  at (1.24,0.8) {$T_{rk,r} (r\ge 3)$};
\node[font=\scriptsize,anchor=west] at (1.424,1.41)
  {$K_n$};
\node[font=\scriptsize,anchor=west] at (1.008,0.09)
  {$K_{a,b}$};
\draw[black!60,thin] (1.006,0.075) -- (1.0015,0.018);
\end{tikzpicture}
\vspace{-5mm}
\caption{The lower bounds on $\bk(G)/\sqrt m$ as
functions of $x=\lambda(G)/\sqrt m$.}
\label{fig:two-bounds}
\end{figure}
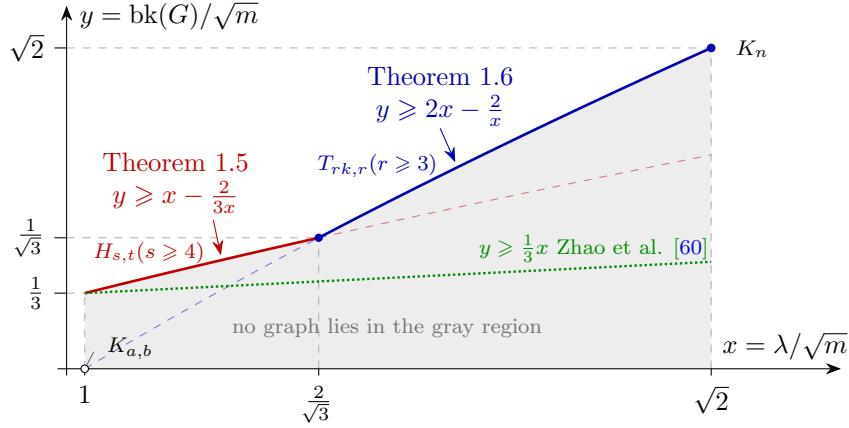

 \subsection{An application: refining the Bollob\'as--Nikiforov bound}

Recall that $t(G)$ denotes the number of triangles in $G$. Bollob\'as and Nikiforov
\cite{BN2007} showed that 
if $G$ is an $m$-edge graph with spectral radius $\lambda$, then 
$$ t(G) \ge \frac{1}{3}\lambda (\lambda^2 - m ).$$ 
 This bound was also proved by 
  Cioab\u{a}, Feng, Tait and Zhang \cite[Lemma 7]{CFTZ20} in an equivalent form: $m \ge \lambda^2 - \frac{3}{\lambda}t(G)$. 
Ning and Zhai \cite{NingZhai2023} provided an alternative proof and showed that equality
holds if and only if $G$ is a complete bipartite graph.  
As an application of our method, we improve the Bollob\'as--Nikiforov bound throughout the full range $\sqrt m\le \lambda < \sqrt{2m}$. 

\begin{theorem}\label{thm:BN-refined}
Every graph $G$ with $m\ge 1$ edges and spectral radius
$\lambda$ satisfies
\begin{equation}\label{eq:BN-refined}
t(G)\ \ge\ \frac{1}{3} (\lambda+1)(\lambda^{2}-m).
\end{equation}
Equality holds if and only if $G$ is a complete graph or a
complete bipartite graph.
\end{theorem}

\subsection{Stability results at the two thresholds}

Once a threshold theorem is known, the natural question is what the near-extremal graphs look like. We
answer it at both thresholds: $\frac16 n$ and
$\frac13\sqrt m$. In each case, 
 the answer is a dichotomy:
a graph whose spectral radius is close to the threshold
either contains a book of almost extremal size, or it is
close in edit distance to the extremal graph. For a
prescribed order, the extremal graph is the bipartite Tur\'{a}n graph $T_{n,2}$ on the same vertex set $V(G)$. For a prescribed size, it is a complete bipartite graph $K_{A,B}$ with $A,B\subseteq V(G)$ (the two parts are not necessarily balanced).

\begin{theorem}\label{thm:edit-distance}
For every $\varepsilon>0$, there is $\delta>0$ such that
the following holds. If $G$ is an $n$-vertex graph with
$\lambda(G)\ge\bigl(\frac12-\delta\bigr)n$, then either
\[
\bk(G)\ge\Bigl(\frac16-\varepsilon\Bigr)n ,
\]
or $G$ can be changed into $T_{n,2}$ by adding and
deleting at most $\varepsilon n^2$ edges. 
\end{theorem}

Theorem \ref{thm:edit-distance} is the spectral
counterpart of the stability theorem of Bollob\'{a}s
and Nikiforov \cite[Theorem 6]{BN2005}. In fact, we
prove a sharper form with explicit exponents (Theorem
\ref{thm:direct}).

 Li, Liu and Zhang \cite[Theorem
4.1]{LLZ-edge-spectral} established a stability result
for Nosal's theorem: for every
$\varepsilon \in (0,0.01)$ there is
$\delta =\delta (\varepsilon )>0$ such that every
triangle-free graph with $m$ edges and
$\lambda (G)\ge (1-\delta )\sqrt m$ admits disjoint
sets $A,B\subseteq V(G)$ with
$|E(G)\,\triangle\, E(K_{A,B})|\le \varepsilon m$. 
The following stability version of \eqref{eq-ZYZZ}
implies it: a triangle-free graph has $\bk (G)=0$, so
the first alternative below fails and the second one
must hold. The proof gives an explicit
threshold 
$\delta =\frac{1}{1000}\min\{\varepsilon ,\frac16\}^3$.

\begin{theorem} \label{thm:edge-spectral-stability}
For every $\varepsilon >0$, there exists
$\delta >0$ such that the following holds. 
If $G$ is a graph with $m\ge 1$ edges and 
$\lambda (G)\ge  (1- \delta) \sqrt{m}$, then either
\[
\bk (G) \ge \Bigl(\frac{1}{3} - \varepsilon \Bigr)
\sqrt{m},
\]
or there exist disjoint sets $A,B\subseteq V(G)$ such
that
$\bigl|E(G)\, {\triangle}\, E(K_{A,B})\bigr|
\le \varepsilon m$.
\end{theorem}

The two statements have different target graphs, and
this is not an artifact of the proofs. Under a
condition on the order, $T_{n,2}$ is the unique
extremal graph, so the balanced complete bipartite
graph is the only possible model. Under a condition on
the size, every $K_{a,b}$ satisfies $\lambda ^2=m$, so
no balance can be forced and the model must be an
arbitrary complete bipartite graph.

A by-product of the proof of Theorem \ref{thm:edge-spectral-stability} is of independent interest.
Lemma \ref{lem:bipartite-completion} shows that every
bipartite graph $G$ satisfies
$\bigl|E(G)\,\triangle\, E(K_{X,Y})\bigr|\le
3(m-\lambda ^2)$ for some $X,Y\subseteq V(G)$; that is, the
edit distance to a complete bipartite graph is
bounded by a linear function of the spectral deficit.
This sharpens the corresponding step of
\cite[Lemma 4.4]{LLZ-edge-spectral}, where the
dependence is of order $\varepsilon ^4$.

\paragraph{Outline of the method.}
Theorems \ref{thm-confirm-Zhai-Lin}, \ref{thm:fixed-order-intro},
\ref{thm:edge-spectral-intro} and \ref{thm:two-lambda} all
come from three relations for a single pair of
Perron-weighted statistics, and the same relations drive
Theorem \ref{thm:BN-refined} and the two stability
theorems. Let $G$ be connected with
$m$ edges, let $\bm{x}$ be a Perron vector of $G$ normalized
by $\|\bm{x}\|_1=1$, and put $\beta =\bk (G)$. For a vertex
$v$, let $t_v$ be the number of triangles through $v$ and let
$s_v$ be the number of edges with both ends outside $N[v]$,
and set
\[
T=\sum_{v\in V(G)}x_vt_v,
\qquad
S=\sum_{v\in V(G)}x_vs_v .
\]
Counting the edges met by $N(v)$ and weighting by $x_v$
yields the identity
\begin{equation}\label{eq:intro-TS}
T-S=\lambda ^2-m,
\end{equation}
which converts the spectral surplus $\lambda ^2-m$ into a
weighted count of triangles. A charging argument over the
triangles meeting $N(v)$ yields
\begin{equation}\label{eq:intro-charge}
\lambda \,T\le \beta \,(2T+S),
\end{equation}
and summing codegrees at each vertex yields
\begin{equation}\label{eq:intro-cap}
T\le \tfrac12\,\beta \lambda .
\end{equation}
Eliminating $S$ between \eqref{eq:intro-TS} and
\eqref{eq:intro-charge} gives
$(3\beta -\lambda )T\ge \beta (\lambda ^2-m)$, a
Perron-weighted analogue of \eqref{eq-BN}. Since $T>0$ whenever $\lambda ^2\ge m$ and $G$ is not complete bipartite,
this alone gives $\beta \ge \frac13\lambda $; in this form the argument is that of Zhao, You, Zeng and Zhang \cite{ZYZZ2026}. 
Feeding \eqref{eq:intro-cap} into it gives Theorem
\ref{thm:edge-spectral-intro}, while \eqref{eq:intro-TS} and \eqref{eq:intro-cap} alone give Theorem
\ref{thm:two-lambda}.  Theorems \ref{thm-confirm-Zhai-Lin} and \ref{thm:fixed-order-intro} then follow from
$\lambda \ge \frac{2m}{n}$, except in the range
$\lambda <\sqrt m$, where $\lambda \ge \lambda (T_{n,2})$ forces $m>\lfloor n^2/4\rfloor$ and the Edwards bound applies directly. The same three relations also drive both stability results.

\paragraph{Organization.} 
The rest of the paper is organized as follows. 
Section \ref{sec:weighted} develops the Perron-weighted identities and the local charging lemma, proves Theorems \ref{thm:edge-spectral-intro}, \ref{thm:two-lambda} and
\ref{thm:BN-refined} together with the equality case of Theorem \ref{thm:edge-spectral-intro} (Proposition
\ref{prop:equality-edge}), and then derives Theorems
\ref{thm-confirm-Zhai-Lin} and \ref{thm:fixed-order-intro} from the edge-spectral bounds.  
Section \ref{sec:sharpness} characterizes the equality case of Theorem \ref{thm-confirm-Zhai-Lin}, presents four families of extremal graphs, and shows that Theorems \ref{thm:fixed-order-intro}, \ref{thm:edge-spectral-intro} and \ref{thm:two-lambda} are best possible. 
Section \ref{sec:stability} proves the two stability results, Theorems \ref{thm:edit-distance} and
\ref{thm:edge-spectral-stability}. Section
\ref{sec:conclusion} collects concluding remarks and open
problems, and Appendix \ref{sec:App} gives an alternative
proof of the Bollob\'as--Nikiforov inequality \eqref{eq-BN}.

\section{A Perron-weighted counting method}

\label{sec:weighted}

This section develops the counting method behind all our
results. It rests on one identity and two inequalities
relating two Perron-weighted quantities $T$ and $S$, which
we define and establish first; the proofs of the main
theorems follow in the remaining subsections.

\subsection{Three basic relations} 
\label{sec:three-ineq}

Throughout this subsection, $G$ is a connected graph with
$m\ge 1$ edges, $\lambda=\lambda(G)$ and $\beta=\bk(G)$,
and $\bm{x}=(x_v)_{v\in V(G)}$ is a Perron vector of
$G$ normalized so that $\|\bm{x}\|_1=1$. Since $G$ is
connected, the Perron--Frobenius theorem gives $x_v>0$ for
every $v\in V(G)$. For an edge $uv$, we write 
$$t_{uv}=|N(u)\cap N(v)|,$$ 
so that $t_{uv}\le \beta$. For a
vertex $v$, we put
\[
t_v=e\bigl(G[N(v)]\bigr),
\qquad
s_v=e\bigl(G[\overline{N[v]} ]\bigr),
\]
so $t_v$ is the number of triangles containing $v$ and
$s_v$ is the number of edges having both ends outside
$N[v] :=N(v)\cup \{v\}$. Finally, we set 
\[
T=\sum_{v\in V(G)}x_v t_v,
\qquad
S=\sum_{v\in V(G)}x_v s_v .
\]
We write $\T(G)$ for the set of triangles of $G$.

\begin{lemma}\label{lem:degree-neighborhood}
For every vertex $v$ of $G$,
\begin{equation}\label{eq:degree-neighborhood}
\sum_{u\in N(v)}d(u)=m+t_v-s_v .
\end{equation}
Consequently,
\begin{equation}\label{eq:TS}
T-S=\lambda^2-m .
\end{equation}
\end{lemma}

\begin{proof}
On the left-hand side of
\eqref{eq:degree-neighborhood} an edge with both ends
in $N(v)$ is counted twice, an edge with both ends
outside $N(v)\cup \{v\}$ is counted not at all, and every other
edge is counted once. This is exactly the count on the
right-hand side.

Multiplying \eqref{eq:degree-neighborhood} by $x_v$ and
summing over $v$ gives
\[
m+T-S
=\sum_{v\in V(G)}x_v\sum_{u\in N(v)}d(u)
=\sum_{u\in V(G)}d(u)\sum_{v\in N(u)}x_v
=\lambda\sum_{u\in V(G)}d(u)x_u ,
\]
where we used $\|\bm{x}\|_1=1$ in the first equality
and $A(G)\bm{x}=\lambda\bm{x}$ in the last. Summing the
eigenvalue equations over all vertices gives
$\sum_u d(u)x_u=\lambda\sum_u x_u=\lambda$, and
\eqref{eq:TS} follows.
\end{proof}

The identity \eqref{eq:TS} and the inequality \eqref{eq:rhoT}
below are, in vertex-local form, the two ingredients of the
proof of $\bk(G)\ge\frac13\lambda (G)$ by Zhao, You, Zeng and
Zhang \cite[Section 3]{ZYZZ2026}, who sum the eigenvalue
equations over the vertices of each triangle; Lemma
\ref{lem:triangle} likewise corresponds to
\cite[Lemma 3.2]{ZYZZ2026}. Both ingredients are
Perron-weighted forms of the four-vertex configuration count
of Bollob\'as and Nikiforov \cite{BN2005}, which goes back to
Khad\v{z}iivanov and Nikiforov \cite{KN1979}; see the
beginning of the proof of \cite[Theorem 1]{BN2005}, and
\cite[Lemma 3]{DNP2026+} for an argument in the same spirit.
The new ingredient is the cap \eqref{eq:T-upper} of Lemma
\ref{lem:T-upper}. The argument of \cite{ZYZZ2026} discards
the surplus $\lambda^2-m$ at the last step; the cap converts
it into a gain in the booksize, and this is what yields
Theorems \ref{thm:edge-spectral-intro} and
\ref{thm:two-lambda}. 

\begin{lemma}[Local charging]\label{lem:charging}
For every vertex $v$ of $G$,
\begin{equation}\label{eq:local}
\sum_{u\in N(v)}t_u\le \beta \,(2t_v+s_v).
\end{equation}
Consequently,
\begin{equation}\label{eq:rhoT}
\lambda\,T\le \beta \,(2T+S).
\end{equation}
\end{lemma}

\begin{proof}
The left-hand side of \eqref{eq:local} counts the pairs
consisting of a triangle and one of its vertices lying
in $N(v)$, so
\begin{equation}\label{eq:inter-step}
\sum_{u\in N(v)}t_u
=\sum_{K\in\T(G)}\bigl|V(K)\cap N(v)\bigr| .
\end{equation}
For $K\in\T(G)$ let $a_K$ be the number of edges of $K$
inside $N(v)$, and let $c_K$ be the number of edges of
$K$ inside $V(G)\setminus (N(v)\cup \{v\})$. We claim that
\begin{equation}\label{eq:pointwise}
\bigl|V(K)\cap N(v)\bigr|\le 2a_K+c_K .
\end{equation}
If $v\in V(K)$, then the other two vertices of $K$ lie
in $N(v)$, and the two sides of \eqref{eq:pointwise}
equal $2$ and $2\cdot 1+0$. If $v\notin V(K)$, put
$k=|V(K)\cap N(v)|$; then $a_K=\binom{k}{2}$ and
$c_K=\binom{3-k}{2}$, so for $k=0,1,2,3$ the right-hand
side of \eqref{eq:pointwise} equals $3,\,1,\,2,\,6$,
each of which is at least $k$.

Summing \eqref{eq:pointwise} over all triangles and
using \eqref{eq:inter-step},
\[
\sum_{u\in N(v)}t_u
\le 2\sum_{K\in\T(G)}a_K+\sum_{K\in\T(G)}c_K .
\]
There are $t_v$ edges inside $N(v)$ and each of them
lies in at most $\beta $ triangles, so
$\sum_{K}a_K\le \beta \,t_v$; likewise
$\sum_{K}c_K\le \beta \,s_v$. This proves \eqref{eq:local}.

Multiplying \eqref{eq:local} by $x_v$ and summing over
$v$, the left-hand side becomes
\[
\sum_{v\in V(G)}x_v\sum_{u\in N(v)}t_u
=\sum_{u\in V(G)}t_u\sum_{v\in N(u)}x_v
=\lambda\sum_{u\in V(G)}t_ux_u=\lambda T ,
\]
while the right-hand side becomes $\beta (2T+S)$. This is
\eqref{eq:rhoT}.
\end{proof}

\begin{lemma}\label{lem:T-upper}
We have
\begin{equation}\label{eq:T-upper}
T\le\frac{\beta \lambda}{2}.
\end{equation}
\end{lemma}

\begin{proof}
For a fixed vertex $v$, the sum of the codegrees over
the edges incident with $v$ counts every triangle
containing $v$ twice, so
$2t_v=\sum_{u\in N(v)}t_{uv}\le \beta\,d(v)$. Multiplying
by $x_v$ and summing over $v$ gives
$2T\le \beta \sum_v d(v)x_v=\beta \lambda$.
\end{proof}

\begin{lemma}\label{lem:triangle}
Suppose that $m\le\lambda^2$. If $G$ is not complete
bipartite, then $T>0$; in particular $G$ contains a
triangle.
\end{lemma}

\begin{proof}
Suppose that $T=0$. Since $x_v>0$ for every $v$, we
have $t_v=0$ for every $v$, so $G$ is triangle-free. By
\eqref{eq:TS}, $S=m-\lambda^2\le 0$, while $S\ge 0$
because every $s_v$ is nonnegative. Hence $S=0$, so $s_w=0$ for every $w\in V(G)$: every edge
of $G$ has an end in $N[w]$. An edge incident with $w$ has  its other end in $N(w)$, and an edge not incident with $w$ meets $N[w]$ only in $N(w)$. Therefore
\begin{equation}\label{eq:hit}
\text{every edge of }G\text{ meets }N(w)
\quad\text{for every }w\in V(G).
\end{equation} 
Fix an edge $uv$ and put $A=N(u)$ and $B=N(v)$. As $G$
is triangle-free, we have 
$A\cap B=\varnothing$ and both $A$
and $B$ are independent; moreover $v\in A$ and $u\in B$.
For every vertex $w$, applying \eqref{eq:hit} to the edge
$uv$ shows that $u\in N(w)$ or $v\in N(w)$, that is,
$w\in N(u)\cup N(v)$. Hence $V(G)=A\cup B$.

Let $a\in A$ and $b\in B$. Then $vb\in E(G)$. If $a=v$,
then $ab=vb\in E(G)$. If $a\ne v$, then $a,v\in A$ and
$A$ is independent, so $v\notin N(a)$; applying
\eqref{eq:hit} to the edge $vb$ gives $b\in N(a)$. In
either case $ab\in E(G)$. Therefore $G=K_{A,B}$ is
complete bipartite, a contradiction.
\end{proof}

Applied to the components of a triangle-free graph, Lemma
\ref{lem:triangle} recovers the equality case of Nosal's
theorem, which is due to Nikiforov \cite{Nikiforov2006walks}:
a triangle-free graph with $m$ edges and $\lambda(G)\ge\sqrt m$
is the disjoint union of a complete bipartite graph and isolated vertices.

\subsection{Proofs of Theorems \ref{thm:edge-spectral-intro} and \ref{thm:two-lambda}}

The proofs in this subsection apply to arbitrary graphs and
determine the extremal graphs.

\begin{proof}[{\bf Proof of Theorem
\ref{thm:edge-spectral-intro}}]
Write $\lambda=\lambda(G)$, choose a component $H$ of
$G$ with $\lambda(H)=\lambda$, and put $h=e(H)$ and
$\beta =\bk(H)$. Then
$ h\le m\le\lambda^2 $.

If $H$ were a complete bipartite graph $K_{p,q}$,
then $\lambda^2=pq=h$, so $h=m$ and every component
of $G$ other than $H$ is a single vertex; that is,
$G$ is the disjoint union of a complete bipartite
graph and isolated vertices, contrary to the
hypothesis. Hence $H$ is not complete bipartite.

Apply the notation above to $H$, and put
$\delta=\lambda^2-h\ge 0$, so that \eqref{eq:TS} reads
$T-S=\delta$. By Lemma \ref{lem:triangle}, $T>0$ and
$H$ contains a triangle, whence $\beta \ge 1$. Substituting
$S=T-\delta$ into \eqref{eq:rhoT} gives
$\lambda T\le \beta (3T-\delta)$, that is,
\begin{equation}\label{eq:first-chain}
(3\beta -\lambda)\,T\ \ge\ \beta \,\delta\ \ge\ 0 .
\end{equation}
Since $T>0$, we obtain $3\beta -\lambda\ge 0$. Multiplying
\eqref{eq:T-upper} by $3\beta -\lambda\ge 0$ and combining
the result with \eqref{eq:first-chain},
\[
\beta \,\delta\ \le\ (3\beta -\lambda)\,T\ \le\
\frac{\beta \lambda(3\beta -\lambda)}{2}.
\]
Cancelling $\beta >0$ gives
$2\lambda^{2}-2h\le 3\beta \lambda-\lambda^{2}$, and
therefore
\[
\bk(G)\ \ge\ \beta \ \ge\
\lambda-\frac{2h}{3\lambda}\ \ge\
\lambda-\frac{2m}{3\lambda},
\]
the last step because $h\le m$.
\end{proof} 

The proof of Theorem
\ref{thm:edge-spectral-intro} yields a
spectral analogue of the Bollob\'as--Nikiforov
inequality \eqref{eq-BN}. Indeed, for a connected
graph $G$, combining Lemmas
\ref{lem:degree-neighborhood} and \ref{lem:charging}
yields
\begin{equation}\label{eq:spectral-BN}
\bigl(3\bk(G)-\lambda\bigr)\,T
\ \ge\ \bk(G)\,\bigl(\lambda^{2}-m\bigr),
\end{equation}
where $T=\sum_{K\in\T(G)}\sum_{v\in V(K)}x_{v}$ is the
Perron-weighted number of triangles. To compare
\eqref{eq:spectral-BN} with \eqref{eq-BN}, note that
summing \eqref{eq:degree-neighborhood} over all vertices
with uniform weights gives
$\sum_{v}t_{v}-\sum_{v}s_{v}=Q-mn$, while
$\sum_{v}t_{v}=3t(G)$. In this sense
\eqref{eq:spectral-BN} is a Perron-weighted analogue of
\eqref{eq-BN}: the counting measure is replaced by the
Perron measure, the average degree $2m/n$ by $\lambda$, and
the degree deficit $Q-mn$ by the spectral deficit
$\lambda^{2}-m$. Moreover,
\eqref{eq:spectral-BN} is sharp: a direct computation
shows that every graph $H_{s,t}$ of Example
\ref{ex:EFR} and the Tur\'an graph $T_{3k,3}$ attain
equality.

\smallskip 
The proof of Theorem \ref{thm:edge-spectral-intro} also
determines the graphs attaining equality above the Nosal
threshold. For a vertex $v$, we write $G-N[v]$ for the graph
obtained from $G$ by deleting $N[v]$.

\begin{proposition}\label{prop:equality-edge}
Let $G$ be a graph with $m\ge1$ edges, and let $\lambda=\lambda(G)$ and $\beta=\bk(G)$.
Suppose that $\lambda^2>m$. Then
\[
\beta=\lambda-\frac{2m}{3\lambda}
\]
if and only if $G$ is connected and $K_4$-free, every edge of
$G$ lies in exactly $\beta$ triangles, and $G-N[v]$ is
triangle-free for every vertex $v$.
\end{proposition}

\begin{proof}
Let $H$ be a component of $G$ with $\lambda(H)=\lambda$, and
put $h=e(H)$. Since $h\le m<\lambda^2=\lambda(H)^2$, the
component $H$ is not complete bipartite, so the proof of
Theorem \ref{thm:edge-spectral-intro} gives 
$\bk(G)\ge\bk(H)\ge\lambda-\frac{2h}{3\lambda}
\ge\lambda-\frac{2m}{3\lambda}$, so equality forces $h=m$.
Then every other component is an isolated vertex, and $G=H$
is connected. Since connectivity is also one of the
conditions on the right-hand side, we may assume that $G$ is
connected. 
Put $\delta=\lambda^2-m>0$. A connected complete bipartite
graph satisfies $\lambda^2=m$, so $G$ is not complete
bipartite, and Lemma \ref{lem:triangle} gives $T>0$ and
$\beta\ge1$. The proof of Theorem
\ref{thm:edge-spectral-intro} yields the chain
\begin{equation}\label{eq:equality-chain}
\beta\delta\ \le\ (3\beta-\lambda)\,T\ \le\
\frac{\beta\lambda(3\beta-\lambda)}{2}.
\end{equation}
Its first inequality is \eqref{eq:rhoT} after the substitution
$S=T-\delta$, and its second inequality is \eqref{eq:T-upper}
multiplied by $3\beta-\lambda\ge0$. Since $\beta>0$, the
equality $\beta=\lambda-\frac{2m}{3\lambda}$ is equivalent to
$2\delta=\lambda(3\beta-\lambda)$, that is, to equality
throughout \eqref{eq:equality-chain}.

Suppose that equality holds throughout
\eqref{eq:equality-chain}. If $3\beta=\lambda$, then
\eqref{eq:equality-chain} gives $\beta\delta\le0$, which is
impossible.  Hence $3\beta>\lambda$, and equality in the second inequality means $T=\frac{\beta\lambda}{2}$, that is, $\sum_{v}x_v\bigl(\beta d(v)-2t_v\bigr)=0$. Each summand is nonnegative, because $2t_v=\sum_{u\in N(v)}t_{uv}\le\beta d(v)$ and $x_v>0$. Hence $2t_v=\beta d(v)$ for every $v$, and so $t_{uv}=\beta$ for every edge $uv$.  Equality in the first inequality is equality in
\eqref{eq:rhoT}, hence in \eqref{eq:local} for every vertex
$v$, because every $x_v$ is positive. In the proof of Lemma
\ref{lem:charging}, \eqref{eq:local} was obtained by summing
\eqref{eq:pointwise} over all triangles and then applying
$\sum_K a_K\le\beta t_v$ and $\sum_K c_K\le\beta s_v$. Hence
\eqref{eq:pointwise} is an equality for every triangle $K$.
For $v\notin V(K)$, the two sides of \eqref{eq:pointwise}
agree only when $k=|V(K)\cap N(v)|\in\{1,2\}$. Thus no vertex
is adjacent to all three vertices of a triangle, so $G$ is
$K_4$-free; and no triangle avoids $N[v]$, so $G-N[v]$ is
triangle-free.

Conversely, suppose that $G$ is connected and  $K_4$-free, that every edge
lies in exactly $\beta$ triangles, and that $G-N[v]$ is
triangle-free for every $v$. The second condition gives
$2t_v=\sum_{u\in N(v)}t_{uv}=\beta d(v)$ for every $v$, so
$T=\frac{1}{2} \beta\lambda$; it also gives
$\sum_K a_K=\beta t_v$ and $\sum_K c_K=\beta s_v$. The first
and third conditions give $k\in\{1,2\}$ for every triangle $K$
and every $v\notin V(K)$, so \eqref{eq:pointwise} is an
equality in every case. Hence \eqref{eq:local} and therefore
\eqref{eq:rhoT} are equalities; that is,
$(3\beta-\lambda)T=\beta\delta$. Together with
$T=\frac{1}{2} \beta\lambda$, this gives
$2\delta=\lambda(3\beta-\lambda)$, which is the desired
equality.
\end{proof}

We now prove Theorem \ref{thm:two-lambda}. 

\begin{proof}[{\bf Proof of Theorem \ref{thm:two-lambda}}]
Choose a component $H$ of $G$ with $\lambda(H)=\lambda$,
and put $h=e(H)$ and $b=\bk(H)$; then $h\le m$ and
$b\le\bk(G)$. Apply the notation of Subsection \ref{sec:three-ineq} to $H$: let $\bm{x}$ be a Perron
vector of $H$ with $\|\bm{x}\|_1=1$, so $x_v>0$ for
every $v\in V(H)$, and let $t_v$, $s_v$, $T$, $S$ be the
associated quantities. Since every $s_v$ is
nonnegative, so is $S$, and \eqref{eq:TS} gives
\begin{equation}\label{eq:two-lambda-lower}
T=S+\lambda^{2}-h\ \ge\ \lambda^{2}-h .
\end{equation}
Combining this with $T\le b\lambda/2$ from
\eqref{eq:T-upper} yields
$\lambda^{2}-h\le b\lambda/2$, that is,
\[
\bk(G)\ \ge\ b\ \ge\ 2\lambda-\frac{2h}{\lambda}
\ \ge\ 2\lambda-\frac{2m}{\lambda},
\]
the last step because $h\le m$. This proves
\eqref{eq:two-lambda}.

Suppose now that equality holds in
\eqref{eq:two-lambda}. Then $h=m$, so every component of
$G$ other than $H$ is a single vertex, and both
\eqref{eq:two-lambda-lower} and \eqref{eq:T-upper} are
equalities.

Equality in \eqref{eq:two-lambda-lower} means $S=0$,
hence $s_v=0$ for every $v\in V(H)$. Thus for every
vertex $v$ the set $V(H)\setminus N_H[v]$ is
independent, that is, $u\not\sim v$ and $w\not\sim v$
imply $u\not\sim w$ for all distinct $u,w$. Therefore
the relation ``$u=w$ or $uw\notin E(H)$'' is an
equivalence relation on $V(H)$; its classes are
independent sets, and any two distinct classes are
completely joined. Hence $H$ is a complete multipartite
graph, say with parts $V_1,\dots,V_r$ of sizes
$n_1,\dots,n_r$, where $r\ge2$ and $n=|V(H)|$.

Equality in \eqref{eq:T-upper} means
$\sum_{v\in V(H)}x_v\bigl(b\,d(v)-2t_v\bigr)=0$. Each
summand is nonnegative because
$2t_v=\sum_{u\in N(v)}t_{uv}\le b\,d(v)$, and $x_v>0$,
so $2t_v=b\,d(v)$ for every $v$; consequently
$t_{uv}=b$ for every edge $uv$ of $H$. For $u\in V_i$
and $v\in V_j$ with $i\ne j$ we have
$t_{uv}=n-n_i-n_j$, so $n_i+n_j$ takes the same value
for all $i\ne j$. If $r=2$ this is no restriction, and
$H$ is complete bipartite; if $r\ge3$, then comparing
the pairs $\{i,j\}$ and $\{i,k\}$ gives $n_j=n_k$ for
all $j\ne k$, so all parts have the same size.

Conversely, let $G$ be complete multipartite with parts of sizes $n_1,\dots,n_r$, where either $r=2$ or all parts have the same size $k$. Then $s_v=0$ for every $v$, so $S=0$ and $T=\lambda^2-m$ by \eqref{eq:TS}. If $r=2$, then $G$ is triangle-free and $T=0=\bk(G)\lambda/2$; if all parts have size $k$, then every edge of $G$ has codegree $n-2k=\bk(G)$, so $2t_v=\bk(G)\,d(v)$ for every
$v$ and again $T=\bk(G)\lambda/2$. In both cases
$\lambda^2-m=\bk(G)\lambda/2$, which is equality in \eqref{eq:two-lambda}. 
\end{proof}

\subsection{Proof of Theorem \ref{thm:BN-refined}}

We show that the weighted identities of Subsection \ref{sec:three-ineq}  have a further application. 
Recall that Bollob\'as and Nikiforov \cite{BN2007} proved that every $m$-edge graph $G$ contains  
$t(G)\ge\frac{1}{3}\lambda(\lambda^{2}-m)$ triangles, and Ning
and Zhai \cite{NingZhai2023} showed that equality
holds if and only if $G$ is a complete bipartite graph. 
Theorem \ref{thm:BN-refined} sharpens the bound
so that complete graphs become extremal as well.

\begin{proof}[{\bf Proof of Theorem \ref{thm:BN-refined}}]
If $\lambda^{2}<m$, then the right-hand side of
\eqref{eq:BN-refined} is negative and the inequality is
trivial. So assume $\lambda^{2}\ge m$. Choose a connected
component $H$
of $G$ with $\lambda(H)=\lambda$, and apply the
notation of Subsection \ref{sec:three-ineq} to $H$: let
$\bm{x}$ be a Perron vector of $H$ with 
$\|\bm{x}\|_{1}=1$, and let $T$ and $S$ be the
associated quantities. Since $S\ge 0$ and
$e(H)\le m$, the identity \eqref{eq:TS} gives
\begin{equation}\label{eq:T-lower-supersat}
T\ \ge\ \lambda^{2}-e(H)\ \ge\ \lambda^{2}-m.
\end{equation}
Grouping the sum $T=\sum_{v}x_{v}t_{v}$ by triangles,
we get $T=\sum_{K\in\T(H)}\sum_{v\in V(K)}x_{v}$, and
hence
\begin{equation}\label{eq:T-xmax}
T\ \le\ 3\,x_{\max}\,t(H),
\qquad\text{where }
x_{\max}:=\max_{v\in V(H)}x_{v}.
\end{equation}
Finally, let $v^{*}$ be a vertex with
$x_{v^{*}}=x_{\max}$. The eigenvalue equation at
$v^{*}$ and $\|\bm{x}\|_{1}=1$ yield
\[
\lambda\,x_{\max}
=\sum_{u\in N(v^{*})}x_{u}
\ \le\ \sum_{u\neq v^{*}}x_{u}
=1-x_{\max},
\]
so $x_{\max}\le\frac{1}{1+\lambda}$. Using \eqref{eq:T-lower-supersat} and
\eqref{eq:T-xmax}, we get 
$\frac{3\,t(H)}{1+\lambda}\ge\lambda^{2}-m$, and
\eqref{eq:BN-refined} follows from $t(G)\ge t(H)$.

Suppose now that equality holds in
\eqref{eq:BN-refined}. Assume first that
$t(G)>0$, so $\lambda^{2}>m$. Then $t(G)=t(H)$,
$e(H)=m$, and \eqref{eq:T-lower-supersat},
\eqref{eq:T-xmax} and
$x_{\max}\le\frac{1}{1+\lambda}$ are all equalities.
From $e(H)=m$, all components of $G$ other than $H$
are isolated vertices. As in the proof of Theorem
\ref{thm:two-lambda}, the equality $S=0$ implies that
$H$ is a complete multipartite graph. Equality in
\eqref{eq:T-xmax} means that $x_{v}=x_{\max}$ for
every vertex $v$ lying in a triangle of $H$; since $H$
is complete multipartite and contains a triangle, it
has at least three parts, so every vertex of $H$ lies
in a triangle. Hence $\bm{x}$ is constant, $H$ is
regular, and all parts of $H$ have the same size.
Equality in $x_{\max}\le\frac{1}{1+\lambda}$ means
that $x_{u}=0$ for every $u\notin N[v^{*}]$; as
$\bm{x}$ is positive on $V(H)$, the vertex $v^{*}$
dominates $H$, so its part is a singleton. Therefore
all parts are singletons and $H$ is a complete graph.
Assume next that $t(G)=0$. Then $\lambda^{2}=m$, and
$H$ is a triangle-free connected graph with
$e(H)\le\lambda^{2}$, so Lemma \ref{lem:triangle}
shows that $H$ is complete bipartite; now
$\lambda^{2}=e(H)$ forces $e(H)=m$, and all remaining
components are isolated vertices.

Conversely, if $G$ is a complete graph $K_{a}$, then $t(G)=\binom{a}{3}$, $\lambda=a-1$ and
$m=\binom{a}{2}$, so both sides of
\eqref{eq:BN-refined} equal $\frac{1}{6}a(a-1)(a-2)$.
If $G$ is a complete bipartite graph, then $t(G)=0$ and
$\lambda^{2}=m$.
\end{proof}

\begin{remark}
Theorem \ref{thm:BN-refined} strengthens the
Bollob\'as--Nikiforov bound whenever $\lambda^{2}>m$,
and it is exact on every complete graph. The bound
$t(G)\ge m(\lambda-\sqrt m\,)$ of Chen, Li and Tang
\cite{ChenLiTang2026} serves a different regime: it is
designed for the Nosal threshold $\lambda\approx\sqrt m$, whereas \eqref{eq:BN-refined} governs the dense range $\lambda\ge 1.31\sqrt m$, where the cubic term
$\lambda(\lambda^{2}-m)$ dominates. 
\end{remark}

\subsection{Proofs of Theorems \ref{thm-confirm-Zhai-Lin} and \ref{thm:fixed-order-intro}}

In this subsection, we prove Theorems \ref{thm-confirm-Zhai-Lin} and \ref{thm:fixed-order-intro} by applying 
Theorems \ref{thm:edge-spectral-intro} and \ref{thm:two-lambda}. 

\begin{proof}[{\bf Proof of Theorem \ref{thm-confirm-Zhai-Lin}}]
We split the proof according to the comparison
between $\lambda$ and $\sqrt m$. 
Suppose first that $\lambda\ge\sqrt m$. We claim that $G$
is not the disjoint union of a complete bipartite graph
and isolated vertices. Indeed, suppose that
$G=K_{a,b}\cup rK_1$. Then $\lambda^2=ab$, and
$\lambda\ge\lambda(T_{n,2})$ gives
\[
\Bigl\lfloor\frac{n^2}{4}\Bigr\rfloor\le ab\le
\Bigl\lfloor\frac{(a+b)^2}{4}\Bigr\rfloor .
\]
If $a+b<n$, then
$\lfloor(a+b)^2/4\rfloor\le\lfloor(n-1)^2/4\rfloor
<\lfloor n^2/4\rfloor$, a contradiction. Hence $a+b=n$,
$r=0$ and $ab=\lfloor n^2/4\rfloor$, so
$\{a,b\}=\{\lceil n/2\rceil,\lfloor n/2\rfloor\}$, that is,
$G=T_{n,2}$, contrary to the hypothesis. Let $G'$ be the
graph obtained from $G$ by deleting its isolated vertices.
Then $G'$ has $m\ge1$ edges, $\lambda(G')=\lambda$ and
$\bk(G')=\bk(G)$, and $G'$ is not complete bipartite by the
claim just proved. 
Thus, Theorem \ref{thm:edge-spectral-intro}
applies to $G'$ and yields
\[
\bk(G)=\bk(G')\ge\lambda-\frac{2m}{3\lambda}
\ge\frac{\lambda}{3}.
\]
In the second case $\lambda<\sqrt m$, the hypothesis
$\lambda\ge\lambda(T_{n,2})=\sqrt{\lfloor n^2/4\rfloor}$
gives $m>\lfloor n^2/4\rfloor$, and the Edwards bound
\eqref{eq-Edw} gives $\bk(G)>n/6$. On the other hand,
$\lambda^2<m\le n\lambda/2$ gives $\lambda<n/2$, so
$\lambda/3<n/6<\bk(G)$. In both cases
$\bk(G)\ge\frac13\lambda(G)$. 

Suppose now that $\lambda(G)>\lambda(T_{n,2})$; then
$G\ne T_{n,2}$, so the first part applies. Put
$q:=\lfloor n/6\rfloor$. Since $3q$ is an integer with
$3q\le n/2$, we have $9q^2\le\lfloor n^2/4\rfloor$, that
is, $\lambda(T_{n,2})=\sqrt{\lfloor n^2/4\rfloor}\ge 3q$.
Therefore
\[
\bk(G)\ \ge\ \tfrac13\lambda(G)
\ >\ \tfrac13\lambda(T_{n,2})\ \ge\ q .
\]
Since $\bk(G)$ is an integer, $\bk(G)\ge q+1
=\lfloor n/6\rfloor+1$, as needed.
\end{proof}

We next derive Theorem \ref{thm:fixed-order-intro} from Theorems \ref{thm:edge-spectral-intro} and \ref{thm:two-lambda}. 

\begin{proof}[{\bf Proof of Theorem \ref{thm:fixed-order-intro}}]
First, suppose that $\lambda \ge \sqrt{m}$. As in
the proof of Theorem \ref{thm-confirm-Zhai-Lin}, the graph obtained
from $G$ by deleting its isolated vertices is not
complete bipartite, and it has the same $m$, $\lambda$
and booksize as $G$. Theorem
\ref{thm:edge-spectral-intro} applied to it gives
\[ \bk (G) \ge 
\lambda-\frac{2m}{3\lambda}\ge \lambda-\frac n3,
\] 
where the second inequality holds since $\lambda\ge2m/n$. 
In the second case $\lambda< \sqrt{m}$, the hypothesis $\lambda \ge \lambda (T_{n,2})$ gives 
$m> \lfloor n^2/4 \rfloor$, and the 
Edwards bound gives $\bk (G)>n/6$. On the other hand, the inequalities $m>\lambda^2$ and $m\le n\lambda/2$ imply $\lambda<n/2$. It follows that $\frac n6 > \lambda-\frac n3$, as needed.

Finally, we apply Theorem \ref{thm:two-lambda} to $G$ with its
isolated vertices deleted, which gives
$\bk(G)\ge 2\lambda-\frac{2m}{\lambda}$. Since
$\lambda\ge\frac{2m}{n}$, we have $\frac{2m}{\lambda}\le n$,
and therefore $\bk(G)\ge 2\lambda-n$. 
\end{proof}

 The second bound has a
direct proof from a degree bound on the spectral radius.

\begin{proof}[Alternative proof]
Let $\beta=\bk(G)$. For any edge $uv\in E(G)$, we have $|N(u)\cap N(v)|\leq \beta$.
Since $|N(u)\cup N(v)| = d(u) + d(v) - |N(u)\cap N(v)|\leq n$, we obtain
$d(u)+d(v)\leq n+\beta$.

A well-known bound on $\lambda(G)$ (see
\cite[Lemma 2.1]{Berman-Zhang2001}) states that
$$\lambda(G)\le\max_{uv\in E(G)}\sqrt{d(u)d(v)}.$$
Combined with the AM--GM inequality
$\sqrt{d(u)d(v)}\le\frac12\bigl(d(u)+d(v)\bigr)$, this gives
\[
\lambda(G)\le\frac12\max_{uv\in E(G)}\bigl(d(u)+d(v)\bigr)
\le\frac{n+\beta}{2}.
\]
Therefore, we get $\beta\ge 2\lambda(G)-n$, as desired. 
\end{proof}

The equality cases of 
Theorems \ref{thm-confirm-Zhai-Lin} and \ref{thm:fixed-order-intro} will be characterized in the next section.

\section{Extremal graphs}\label{sec:sharpness}

In this section, we determine exactly which graphs attain
equality in Theorem \ref{thm-confirm-Zhai-Lin} (Section \ref{sec:tight}). These graphs satisfy $\lambda(G)=\lambda(T_{n,2})$, so the
strict inequality in the second statement of Theorem
\ref{thm-confirm-Zhai-Lin} cannot be relaxed. 
We then present four infinite families of such graphs and show that their number grows at least linearly with $n$ (Section \ref{sec:families}). 
By perturbing these graphs, we show
that the bound $\lfloor \frac{n}{6} \rfloor+1$ in Theorem
\ref{thm-confirm-Zhai-Lin} is attained (Section \ref{sec:unbalanced}), and we list some tight graphs (Section \ref{sec:small}). 
Along the way, we show that the bounds of Theorems \ref{thm:fixed-order-intro}, \ref{thm:edge-spectral-intro} and \ref{thm:two-lambda} are best possible as well.

\subsection{Tight graphs}\label{sec:tight}

Let $t_{uv}=|N(u)\cap N(v)|$ denote the
number of triangles containing the edge $uv$.

\begin{definition}\label{def:tight}
A graph $G$ on $n$ vertices is \emph{tight} if $G$ is
$\frac n2$-regular, $G$ contains a triangle, and every edge of
$G$ lies in either $0$ or exactly $\frac n6$ triangles.
\end{definition}

A tight graph has $n\equiv 0\pmod 6$ and $\frac{n^2}{4}$ edges,
so $\lambda(G)=\frac n2=\lambda(T_{n,2})$ and $G\ne T_{n,2}$.

\begin{lemma}\label{lem:count}
Let $G$ be an $\frac n2$-regular graph on $n$ vertices and let
$uvw$ be a triangle of $G$. Then
\begin{equation}\label{eq:count}
t_{uv}+t_{uw}+t_{vw}\ \ge\ \frac n2 ,
\end{equation}
with equality if and only if every vertex outside the triangle
is adjacent to exactly one or exactly two of $u,v,w$.
Thus, if $G$ contains a triangle, then
$\bk(G)\ge\frac n6$, with equality if and only if $G$ is
tight.
\end{lemma}

\begin{proof}
For an edge $xy$ let $S_{xy}$ be the set of vertices adjacent to
both $x$ and $y$ or to neither of them. Since
$d(x)=d(y)=\frac n2$, we have $|N(x)\cup N(y)|=n-t_{xy}$, and
therefore $|S_{xy}|=2t_{xy}$. Each of $u,v,w$ lies in exactly one
of the sets $S_{uv},S_{uw},S_{vw}$. A vertex $z\notin\{u,v,w\}$
adjacent to exactly $j$ of $u,v,w$ lies in
$\binom j2+\binom{3-j}2$ of them; this number is $3$ for
$j\in\{0,3\}$ and $1$ for $j\in\{1,2\}$. Hence
$|S_{uv}|+|S_{uw}|+|S_{vw}|\ge n$, with equality exactly when
$j\in\{1,2\}$ for every $z$. This proves \eqref{eq:count}.

If $G$ contains a triangle, then \eqref{eq:count} gives
$3\bk(G)\ge\frac n2$. If $\bk(G)=\frac n6$, then \eqref{eq:count}
forces $t_{e}=\frac n6$ for every edge $e$ of every triangle, so
$G$ is tight. Conversely, in a tight graph every edge lies in at
most $\frac n6$ triangles, and \eqref{eq:count} applied to any
triangle shows that some edge lies in exactly $\frac n6$ of them.
\end{proof}

We now characterize the equality case of the bound in  Theorem \ref{thm-confirm-Zhai-Lin}.

\begin{proposition}\label{prop:equality}
Let $G\ne T_{n,2}$ be an $n$-vertex graph with
$\lambda(G)\ge\lambda(T_{n,2})$. Then $\bk(G)=\frac13\lambda(G)$
if and only if $G$ is tight. In particular, equality in Theorem
\ref{thm-confirm-Zhai-Lin} forces $\lambda(G)=\lambda(T_{n,2})$.
\end{proposition}

\begin{proof}
A tight graph satisfies $\lambda(G)=\frac n2=\lambda(T_{n,2})$,
is not $T_{n,2}$, and has $\bk(G)=\frac n6$ by Lemma
\ref{lem:count}.
Conversely, let $\bk(G)=\frac13\lambda$ with $\lambda=\lambda(G)$.
In the proof of Theorem \ref{thm-confirm-Zhai-Lin}, the case $\lambda<\sqrt m$
gives $\lambda<\frac n2$ and $\bk(G)>\frac n6>\frac13\lambda$,
and the case $\lambda\ge\sqrt m$ gives
$\bk(G)\ge\lambda-\frac{2m}{3\lambda}\ge\frac13\lambda$, where
the second step is an equality only if $m=\lambda^2$. Hence, we obtain 
$m=\lambda^2$. Now $\lambda\ge\frac{2m}{n}=\frac{2\lambda^2}{n}$
yields $\lambda\le\frac n2$ and $m=\lambda^2\le\frac{n^2}{4}$,
and $\lambda^2\ge\lfloor n^2/4\rfloor$ yields
$m=\lfloor n^2/4\rfloor$.
Since $G\ne T_{n,2}$ and $T_{n,2}$ is the only triangle-free
graph with $n$ vertices and $\lfloor n^2/4\rfloor$ edges, $G$
has a triangle, so $\bk(G)\ge 1$. If $n$ were odd, then
$\bk(G)=\frac13\lambda=\frac16\sqrt{n^2-1}$ would be a positive
integer $k$, so $(n-6k)(n+6k)=1$, which is impossible. Thus $n$
is even, $m=\frac{n^2}{4}$ and $\lambda=\frac n2=\frac{2m}{n}$,
so $G$ is $\frac n2$-regular and
$\lambda(G)=\frac n2=\lambda(T_{n,2})$. By Lemma
\ref{lem:count}, we see that $\bk(G)=\frac n6$ forces $G$ to be tight.
\end{proof}

\begin{lemma}\label{lem:structure}
Let $G$ be a tight graph on $n=6\beta$ vertices and let $uvw$ be
a triangle of $G$.
\begin{enumerate}
\item[(a)] Every vertex outside the triangle is adjacent to
exactly one or exactly two of $u,v,w$. In particular, $G$ is
$K_4$-free.
\item[(b)] The sets $W=N(v)\cap N(w)$ and
$R=V(G)\setminus(N(v)\cup N(w))$ both have exactly $\beta$
vertices, $u\in W$, and every vertex of $R$ is adjacent to every
vertex of $W$.
\end{enumerate}
\end{lemma}

\begin{proof}
Every edge of the triangle lies in a triangle, hence in exactly
$\beta$ triangles, so equality holds in \eqref{eq:count}. This
gives (a) for the triangle $uvw$. Since $uvw$ was arbitrary, no
vertex is joined to all three vertices of any triangle, and $G$
is $K_4$-free. For (b), we have $|W|=t_{vw}=\beta$ and
$|R|=n-|N(v)\cup N(w)|=\beta$. Let $r\in R$ and $u'\in W$. Then
$u'vw$ is a triangle and $r$ is adjacent to neither $v$ nor $w$,
so $r$ is adjacent to $u'$ by (a).
\end{proof}

\begin{lemma}\label{lem:parity}
Let $G$ be a tight graph on $n=6\beta$ vertices in which every
edge lies in a triangle. Then $\beta$ is even, that is,
$n\equiv 0\pmod{12}$.
\end{lemma}

\begin{proof}
Fix a triangle $uvw$. By Lemma \ref{lem:structure}(a), the
vertices outside the triangle split into six sets: for
$x\in\{u,v,w\}$, let $Y_x$ be the set of vertices adjacent to $x$
only, and let $X_x$ be the set of vertices adjacent to the two vertices other than $x$. 
In the notation of Lemma
\ref{lem:structure}(b), applied to the edge $vw$, we have $Y_u=R$ and, since $G$ is $K_4$-free,  $X_u=W\setminus\{u\}$. Hence, we get 
$|Y_u|=\beta$, $|X_u|=\beta-1$, and $Y_u$ is completely joined to
$X_u\cup\{u\}$. Moreover, we have $N(u)=\{v,w\}\cup X_v\cup X_w\cup Y_u$.

Let $y\in Y_u$. The edge $uy$ lies in a triangle, so
$t_{uy}=\beta$. As $y$ is adjacent to neither $v$ nor $w$, the
$\beta$ common neighbors of $u$ and $y$ lie in
$X_v\cup X_w\cup Y_u$. Together with $u$ and the $\beta-1$
vertices of $X_u$, this accounts for $2\beta$ neighbors of $y$;
the remaining $\beta$ neighbors of $y$ lie in $Y_v\cup Y_w$.
Hence $e(Y_u,Y_v)+e(Y_u,Y_w)=\beta^{2}$, and the same holds
with $u$ replaced by $v$ or by $w$. Adding the three equations
gives
$2\bigl(e(Y_u,Y_v)+e(Y_u,Y_w)+e(Y_v,Y_w)\bigr)=3\beta^{2}$;
halving this and subtracting the equation for $u$ leaves
$e(Y_v,Y_w)=\beta^{2}/2$. Therefore $\beta^{2}$ is even, and
so is $\beta$.
\end{proof}

Consequently, for $n\equiv 6\pmod{12}$, 
Lemma \ref{lem:parity} implies that every tight graph of order $n$ has an
edge $uy$ that lies in no triangle. For such an edge
$N(u)\cap N(y)=\emptyset$, and since $d(u)+d(y)=n$, the ends of
the edge have complementary neighborhoods:
$N(y)=V(G)\setminus N(u)$.

Now, the graphs attaining equality in either term of
Theorem \ref{thm:fixed-order-intro} can be described.

\begin{proposition}\label{prop:equality-fixed}
Let $G\ne T_{n,2}$ be an $n$-vertex graph with
$\lambda:=\lambda(G)\ge\lambda(T_{n,2})$.
\begin{enumerate}
\item[\rm(a)] $\bk(G)=2\lambda-n$ if and only if $G=T_{n,r}$
for some $r\ge3$ with $r\mid n$.
\item[\rm(b)] $\bk(G)=\lambda-\frac n3$ if and only if $G$ is
$d$-regular for some $d\ge\frac n2$ and, in addition, either
$d=\frac n2$, $G$ has a triangle and every edge of $G$ lies in $0$ or $\frac{n}{6}$ triangles, 
or $d>\frac n2$ and $G$ is
$K_4$-free, every edge of $G$ lies in exactly $\bk(G)$
triangles, and $G-N[v]$ is triangle-free for every vertex $v$.
\end{enumerate}
\end{proposition}

\begin{proof}
Throughout we use $\lambda\ge\frac{2m}{n}$, with equality if
and only if $G$ is regular. It gives $\frac{2m}{\lambda}\le n$
and $\frac{2m}{3\lambda}\le\frac n3$, and each of these two
inequalities is an equality exactly when $G$ is regular. Note
also that $\lambda\ge\lambda(T_{n,2})>0$, so a regular $G$ has
positive degree and hence no isolated vertex.

For (a), Theorem \ref{thm:two-lambda} gives
$\bk(G)\ge2\lambda-\frac{2m}{\lambda}\ge2\lambda-n$. If
$\bk(G)=2\lambda-n$, then $G$ is regular and equality holds in
Theorem \ref{thm:two-lambda}, so $G$ is complete bipartite or
regular complete multipartite. A regular complete bipartite
graph on $n$ vertices is $T_{n,2}$, which is excluded; hence
$G=T_{n,r}$ for some $r\ge3$ with $r\mid n$. Conversely,
$\lambda(T_{n,r})=\frac{r-1}{r}n>\frac n2\ge\lambda(T_{n,2})$
and $\bk(T_{n,r})=2\lambda(T_{n,r})-n$.

For (b), suppose that $\bk(G)=\lambda-\frac n3$. In the proof
of Theorem \ref{thm-confirm-Zhai-Lin}, the case $\lambda<\sqrt m$ gives
$\bk(G)>\frac n6>\lambda-\frac n3$; hence $\lambda\ge\sqrt m$,
and Theorem \ref{thm:edge-spectral-intro}, applied to $G$ with
its isolated vertices deleted, gives
$\bk(G)\ge\lambda-\frac{2m}{3\lambda}\ge\lambda-\frac n3$.
Therefore $G$ is regular, say of degree $d$, and
$\bk(G)=\lambda-\frac{2m}{3\lambda}$. Now $\lambda=d$ and
$m=\frac{nd}{2}$, so $\lambda^{2}-m=d\bigl(d-\frac n2\bigr)$,
and $\lambda\ge\sqrt m$ forces $d\ge\frac n2$. If
$d=\frac n2$, then $\lambda=\frac n2=\lambda(T_{n,2})$ and
$\bk(G)=\frac n6=\frac13\lambda$, so $G$ is tight by
Proposition \ref{prop:equality}. If $d>\frac n2$, then
$\lambda^{2}>m$ and Proposition \ref{prop:equality-edge}
yields the three stated conditions. 
Conversely, a tight graph satisfies
$\bk(G)=\frac n6=\lambda-\frac n3$. Suppose that $G$ is
$d$-regular with $d>\frac n2$ and satisfies the three conditions. Every component of a $d$-regular graph has at least $d+1>\frac n2+1$ vertices, so $G$ is connected.
Moreover, $\lambda=d$ and $m=\frac{nd}{2}<d^{2}$ give
$\lambda^{2}>m$, and Proposition
\ref{prop:equality-edge} gives 
$\bk(G)=\lambda-\frac{2m}{3\lambda}=\lambda-\frac n3$.
\end{proof}

\subsection{Four families of tight graphs}\label{sec:families}

We now describe four infinite families of tight graphs; Figure
\ref{fig:three-constructions} shows one member of each. Three
of the families are balanced blow-ups of a fixed graph, and the
fourth is a two-parameter family that contains the first. Recall that the \emph{blow-up} of a graph $F$ with respect to
nonnegative integers $(w_v)_{v\in V(F)}$ is obtained by
replacing each vertex $v$ by an independent set of size $w_v$,
so that $v$ is deleted when $w_v=0$, and joining two new
vertices when the corresponding vertices of $F$ are
adjacent; the blow-up is \emph{balanced}, or the
\emph{$t$-blow-up}, if all $w_v$ are equal to $t$.

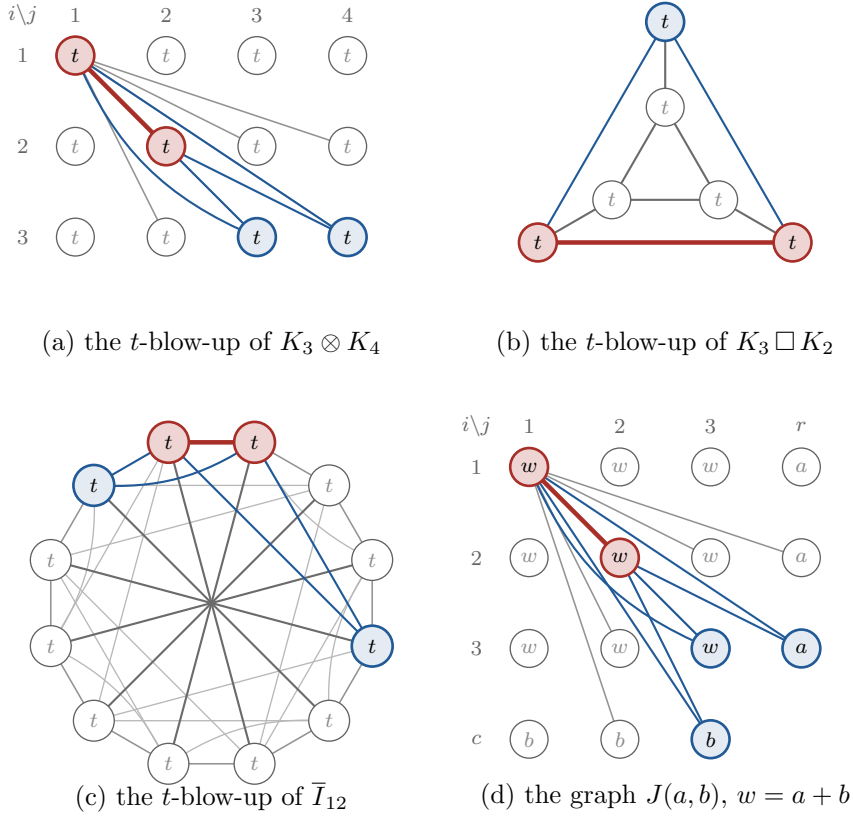
\begin{figure}[h]
\centering
\begin{tikzpicture}[
  x=1cm, y=1cm,
  cls/.style={circle, draw=black!62, fill=white,
              minimum size=5mm, inner sep=0pt,
              font=\scriptsize, line width=0.5pt},
  spine/.style={cls, draw=bookcol, fill=bookcol!20,
                line width=1pt},
  page/.style={cls, draw=pagecol, fill=pagecol!12,
               line width=1pt},
  soft/.style={draw=black!42, line width=0.6pt},
  full/.style={draw=black!58, line width=0.8pt},
  faint/.style={draw=black!28, line width=0.5pt},
  spineedge/.style={draw=bookcol, line width=1.7pt},
  pageedge/.style={draw=pagecol, line width=0.8pt},
  lbl/.style={font=\scriptsize, text=black!60},
  cpt/.style={font=\small}
]

\begin{scope}[scale=1.2]
  \foreach \i in {1,2,3}{
    \foreach \j in {1,2,3,4}{
      \coordinate (v\i\j) at (\j,-\i);
    }
  }
  \foreach \k in {2,3}{
    \foreach \l in {2,3,4}{
      \draw[soft] (v11) -- (v\k\l);
    }
  }
  \draw[pageedge] (v11) to[bend right=22] (v33);
  \draw[pageedge] (v11) -- (v34);
  \draw[pageedge] (v22) -- (v33);
  \draw[pageedge] (v22) -- (v34);
  \draw[spineedge] (v11) -- (v22);
  \foreach \i in {1,2,3}{
    \foreach \j in {1,2,3,4}{
      \node[cls] at (v\i\j) {\color{black!45}$t$};
    }
  }
  \node[spine] at (v11) {$t$};
  \node[spine] at (v22) {$t$};
  \node[page]  at (v33) {$t$};
  \node[page]  at (v34) {$t$};
  \foreach \j in {1,2,3,4}{
    \node[lbl] at (\j,-0.55) {$\j$};
  }
  \foreach \i in {1,2,3}{
    \node[lbl] at (0.42,-\i) {$\i$};
  }
  \node[lbl] at (0.42,-0.55) {$i\backslash j$};
\end{scope}

\begin{scope}[shift={(9.0,-2.7)}, scale=1.2]
  \foreach \a in {90,210,330}{
    \coordinate (A\a) at (\a:1.62);
    \coordinate (B\a) at (\a:0.68);
  }
  \draw[full] (B90) -- (B210) -- (B330) -- cycle;
  \foreach \a in {90,210,330}{
    \draw[full] (A\a) -- (B\a);
  }
  \draw[pageedge] (A90) -- (A210);
  \draw[pageedge] (A90) -- (A330);
  \draw[spineedge] (A210) -- (A330);
  \foreach \a in {90,210,330}{
    \node[cls] at (B\a) {\color{black!45}$t$};
    \node[cls] at (A\a) {\color{black!45}$t$};
  }
  \node[spine] at (A210) {$t$};
  \node[spine] at (A330) {$t$};
  \node[page]  at (A90)  {$t$};
\end{scope}

\begin{scope}[shift={(3.0,-8.44)}, scale=1.1,
              cls/.append style={minimum size=5.4mm}]
  \foreach \i in {0,...,11}{
    \pgfmathsetmacro\ang{105-30*\i}
    \coordinate (P\i) at (\ang:2);
  }
  \foreach \i in {1,...,10}{
    \pgfmathtruncatemacro\j{\i+1}
    \draw[soft] (P\i) -- (P\j);
  }
  \foreach \i in {1,3,5,7,9}{
    \pgfmathtruncatemacro\j{\i+2}
    \draw[faint] (P\i) to[bend right=18] (P\j);
  }
  \draw[faint] (P4) -- (P8);   \draw[faint] (P8) -- (P0);
  \draw[faint] (P2) -- (P6);   \draw[faint] (P6) -- (P10);
  \draw[faint] (P10) -- (P2);
  \draw[faint] (P3) -- (P6);   \draw[faint] (P5) -- (P8);
  \draw[faint] (P7) -- (P10);  \draw[faint] (P9) -- (P0);
  \draw[faint] (P11) -- (P2);
  \foreach \i in {0,...,5}{
    \pgfmathtruncatemacro\j{\i+6}
    \draw[full] (P\i) -- (P\j);
  }
  \draw[pageedge] (P0) -- (P11);
  \draw[pageedge] (P0) -- (P4);
  \draw[pageedge] (P1) -- (P4);
  \draw[pageedge] (P11) to[bend right=18] (P1);
  \draw[spineedge] (P0) -- (P1);
  \foreach \i in {0,...,11}{
    \node[cls] at (P\i) {\color{black!45}$t$};
  }
  \node[spine] at (P0)  {$t$};
  \node[spine] at (P1)  {$t$};
  \node[page]  at (P4)  {$t$};
  \node[page]  at (P11) {$t$};
\end{scope}

\begin{scope}[shift={(6.0,-5.44)}, scale=1.2]
  \foreach \i in {1,2,3}{
    \foreach \j in {1,2,3}{
      \coordinate (u\i\j) at (\j,-\i);
    }
    \coordinate (r\i) at (4,-\i);
    \coordinate (c\i) at (\i,-4);
  }
  \foreach \k in {2,3}{
    \foreach \l in {2,3}{
      \draw[soft] (u11) -- (u\k\l);
    }
  }
  \draw[soft] (u11) -- (r2);
  \draw[soft] (u11) -- (c2);
  \draw[pageedge] (u11) to[bend right=22] (u33);
  \draw[pageedge] (u11) -- (r3);
  \draw[pageedge] (u11) -- (c3);
  \draw[pageedge] (u22) -- (u33);
  \draw[pageedge] (u22) -- (r3);
  \draw[pageedge] (u22) -- (c3);
  \draw[spineedge] (u11) -- (u22);
  \foreach \i in {1,2,3}{
    \foreach \j in {1,2,3}{
      \node[cls] at (u\i\j) {\color{black!45}$w$};
    }
    \node[cls] at (r\i) {\color{black!45}$a$};
    \node[cls] at (c\i) {\color{black!45}$b$};
  }
  \node[spine] at (u11) {$w$};
  \node[spine] at (u22) {$w$};
  \node[page]  at (u33) {$w$};
  \node[page]  at (r3)  {$a$};
  \node[page]  at (c3)  {$b$};
  \foreach \j in {1,2,3}{
    \node[lbl] at (\j,-0.55) {$\j$};
  }
  \node[lbl] at (4,-0.55) {$r$};
  \foreach \i in {1,2,3}{
    \node[lbl] at (0.42,-\i) {$\i$};
  }
  \node[lbl] at (0.42,-4) {$c$};
  \node[lbl] at (0.42,-0.55) {$i\backslash j$};
\end{scope}

\node[cpt] at (3.0,-5)  {(a) the $t$-blow-up of $K_3\otimes K_4$};
\node[cpt] at (9.0,-5)  {(b) the $t$-blow-up of $K_3\,\square\, K_2$};
\node[cpt] at (3.0,-11.0) {(c) the $t$-blow-up of $\overline{I}_{12}$};
\node[cpt] at (9.0,-11.0) {(d) the graph $J(a,b)$, $w=a+b$};

\end{tikzpicture}
\caption{One member of each of the four families of tight
graphs: the $t$-blow-ups $H_{4,t}$, $Y_{t,t}$ and
$\overline{I}_{12}(t)$ of $K_3\otimes K_4$, of the triangular
prism and of $\overline{I}_{12}$, and the graph $J(a,b)$ of
Example \ref{ex:J}. The graphs in (b), (c) and (d) generate
every tight graph on at most $42$ vertices (Subsection \ref{sec:small}).}
\label{fig:three-constructions}
\end{figure}

In Figure \ref{fig:three-constructions}, each circle is an independent set whose size is indicated, and two adjacent
circles form a complete bipartite graph. In (a) and (d), only
the edges at the class in position $(1,1)$ are drawn; in (d),
the class in row $i$ of the column marked $r$ is joined to every
class outside row $i$, and the class in column $j$ of the row
marked $c$ is joined to every class outside column $j$. In (c),
the twelve classes are placed so that all $36$ edges are shown.
In all four graphs, the thick red edge together with the
vertices of the blue classes forms a book of size exactly
$n/6$.

\smallskip
The first family is due to Erd\H{o}s, Faudree and Rousseau
\cite{EFR1994}; see also \cite{BN2005}.

\begin{example}\label{ex:EFR}
Let $s\ge4$ and $t\ge1$ be integers. Partition the vertex set
into $3s$ classes
\[
V_{ij},\qquad i\in\{1,2,3\},\quad j\in\{1,\dots,s\},
\]
each of size $t$. Join a vertex in $V_{ij}$ to a vertex in
$V_{k\ell}$ if and only if $i\ne k$ and $j\ne\ell$. Denote the
resulting graph by $H_{s,t}$. Equivalently, $H_{s,t}$ is the
$t$-blow-up of the tensor product $K_3\otimes K_s$.
\end{example}

The graph has $n=3st$ vertices. Every vertex has $2(s-1)t$
neighbors: there are two choices for the other row index, $s-1$
choices for a different column index, and $t$ vertices in the
resulting class. Thus $H_{s,t}$ is $2(s-1)t$-regular, and
$\lambda(H_{s,t})=2(s-1)t$. The handshaking lemma gives
\begin{equation*}
m=\frac{n\lambda(H_{s,t})}{2}=3s(s-1)t^2 .
\end{equation*}
Moreover,
\begin{equation*}
\lambda(H_{s,t})^2-m=(s-1)(s-4)t^2\ge 0,
\end{equation*}
so the family lies in the range of Theorem
\ref{thm:edge-spectral-intro}.

Consider an edge joining a vertex in $V_{ij}$ to a vertex in
$V_{k\ell}$. We have $i\ne k$ and $j\ne\ell$. A common neighbor
must lie in the third row and in a column different from both
$j$ and $\ell$. There are $s-2$ such columns and $t$ vertices in
each corresponding class. Therefore every edge lies in exactly
$(s-2)t$ triangles, and $\bk(H_{s,t})=(s-2)t$. Moreover,
\begin{align*}
\lambda(H_{s,t})-\frac{2m}{3\lambda(H_{s,t})}
=2(s-1)t-\frac{2\cdot3s(s-1)t^2}{3\cdot2(s-1)t}
=(s-2)t=\bk(H_{s,t})
\end{align*}
and
\begin{align*}
\lambda(H_{s,t})-\frac n3
=2(s-1)t-st=(s-2)t=\bk(H_{s,t}).
\end{align*}
Thus $H_{s,t}$ attains equality in Theorem
\ref{thm:edge-spectral-intro} and in the first term of Theorem
\ref{thm:fixed-order-intro}.

In the case $s=4$, we have $\lambda(H_{4,t})=n/2$,
$m=\lambda^2$ and $\bk(H_{4,t})=n/6=\lambda/3$; that is,
$H_{4,t}$ is tight and attains equality in Theorem
\ref{thm-confirm-Zhai-Lin} as well. For $s>4$, we have
\[
\lambda(H_{s,t})-\frac n2=\frac{(s-4)t}{2}>0 ,
\]
and $\lambda(T_{n,2})\le n/2$, so
$\lambda(H_{s,t})>\lambda(T_{n,2})$ while equality still holds
in $\bk(H_{s,t})=\lambda(H_{s,t})-n/3$. Consequently, neither
Theorem \ref{thm:edge-spectral-intro} nor the term
$\lambda-n/3$ can be increased by any positive universal
amount, even under the strict spectral Tur\'an condition.

\begin{remark}\label{rem:equality-edge}
As shown in Example \ref{ex:EFR}, the graphs $H_{s,t}$ with
$s\ge5$ attain equality in Theorem
\ref{thm:edge-spectral-intro} above the Nosal threshold. A
second family consists of the joins $G=F\vee\overline{K}_r$,
where $F$ is an $r$-regular triangle-free graph on $N\le6r$
vertices and every vertex of $F$ is joined to each of $r$ new
pairwise nonadjacent vertices. Here $m=\frac32Nr$, every edge
lies in exactly $r$ triangles, and $\lambda$ is the larger
root of $\lambda^2-r\lambda-Nr=0$, because the partition into
$V(F)$ and the new vertices is equitable. Hence
$\lambda-\frac{2m}{3\lambda}=\lambda-\frac{Nr}{\lambda}=r
=\bk(G)$, and $\lambda^2-m=r(\lambda-\frac N2)$ is
nonnegative precisely when $N\le6r$, with equality for $N=6r$.
The Tur\'an graph $T_{3k,3}=K_{k,k}\vee\overline{K}_k$ is the
case $F=K_{k,k}$; other examples are $C_5\vee\overline{K}_2$
and the join of the Petersen graph with $\overline{K}_3$.
\end{remark}

\smallskip
The second family is tight for every $n\equiv 0\pmod 6$.

\begin{example}\label{ex:prism}
Let $Y=K_3\,\square\, K_2$ be the triangular prism,
with triangles $a_1a_2a_3$ and $b_1b_2b_3$ and edges
$a_ib_i$ for $i\in\{1,2,3\}$. For positive integers
$p$ and $q$, let $Y_{p,q}$ be obtained from $Y$ by
replacing each $a_i$ by an independent set $A_i$ of
size $p$, each $b_i$ by an independent set $B_i$ of
size $q$, and each edge by a complete bipartite graph.
\end{example}

Put $n=3(p+q)$. Counting the edges inside the two
blown-up triangles and along the matching,
we have $m=3p^2+3pq+3q^2=\frac{1}{4} n^2+
\frac{3}{4} (p-q)^2$.
A vertex of $A_1$ has neighborhood
$A_2\cup A_3\cup B_1$ and a vertex of $B_1$ has
neighborhood $B_2\cup B_3\cup A_1$. These two sets are
disjoint, so the edges between $A_i$ and $B_i$ lie in
no triangle. An edge between $A_i$ and $A_j$ has
common neighborhood the third set $A_\ell$, and
similarly for the $B_i$. Hence $\bk\bigl(Y_{p,q}\bigr)=\max\{p,q\}$. 
For $p=q=k$, the graph $Y_{k,k}$ is $3k$-regular on $n=6k$
vertices, every edge inside a blown-up triangle lies in exactly
$k=\frac n6$ triangles, and every edge between $A_i$ and $B_i$
lies in none. Hence $Y_{k,k}$ is tight. It is not isomorphic to
$H_{4,t}$: taking $k=2t$ so that the orders agree, we have
$t(Y_{2t,2t})=16t^3$ and $t(H_{4,t})=24t^3$.

Next, we present the third family, which is tight for every 
$n\equiv 0\pmod {12}$. 

\begin{example}\label{ex:icosa}
Let $I_{12}$ denote the icosahedron and let $\overline{I}_{12}$
be its complement, a $6$-regular graph on $12$ vertices with
$36$ edges and $20$ triangles. For $t\ge1$ let
$\overline{I}_{12}(t)$ be its $t$-blow-up.
\end{example}

The six diameters of $\overline{I}_{12}$, which join antipodal
vertices of the icosahedron, lie in no triangle, and each of
the other $30$ edges lies in exactly two triangles. Hence
$\overline{I}_{12}(t)$ is $6t$-regular on $n=12t$ vertices,
every edge of it lies in $0$ or in $2t=\frac n6$ triangles, and
$\overline{I}_{12}(t)$ is tight. It is isomorphic neither to
$H_{4,t}$ nor to $Y_{2t,2t}$: the graphs $H_{4,t}$,
$\overline{I}_{12}(t)$ and $Y_{2t,2t}$ have $24t^3$, $20t^3$
and $16t^3$ triangles, respectively.

\smallskip
The fourth family of  graphs $J(a,b)$ is defined as follows. Its base graph $J^*$ contains
$K_3\otimes K_4$ as an induced subgraph, and the family contains
the graphs $H_{4,t}$.

\begin{example} \label{ex:J}
Let $J^*$ be the graph with vertex set
$[3]\times[3]\cup\{r_1,r_2,r_3\}\cup\{c_1,c_2,c_3\}$ in which
$(i,j)$ and $(k,\ell)$ are adjacent if and only if $i\ne k$ and
$j\ne \ell$, the vertex $r_i$ is adjacent to every $(k,\ell)$ with
$k\ne i$, the vertex $c_j$ is adjacent to every $(k,\ell)$ with
$\ell\ne j$, and $r_1,r_2,r_3,c_1,c_2,c_3$ are pairwise nonadjacent.
Thus $[3]\times[3]$ induces $K_3\otimes K_3$, and $r_i$ and $c_j$
are joined to the six vertices outside row $i$ and outside
column $j$ respectively. For $a,b\ge 0$ with
$w:=a+b$, let $J(a,b)$ be the blow-up of $J^*$ in which
every $(i,j)$ is replaced by $w$ vertices, every $r_i$ by $a$
vertices and every $c_j$ by $b$ vertices.
\end{example}

The graph $J(a,b)$ has $n=9w+3a+3b=12w$ vertices. A vertex
arising from $(i,j)$ has $4w$ neighbors arising from the four
vertices $(k,\ell)$ with $k\ne i$ and $\ell\ne j$, $2a$ neighbors
arising from $r_k$ with $k\ne i$, and $2b$ neighbors arising
from $c_{\ell}$ with $\ell\ne j$, that is, $6w$ neighbors in all; a
vertex arising from $r_i$ or from $c_j$ has $6w$ neighbors as
well. An edge between vertices arising from $(i,j)$ and $(k,\ell)$
has $w+a+b=2w$ common neighbors, namely those arising from the
unique $(i',j')$ with $i'\notin\{i,k\}$ and $j'\notin\{j,\ell\}$,
from $r_{i'}$ and from $c_{j'}$. An edge between vertices arising
from $(i,j)$ and from $r_k$ with $k\ne i$ has $2w$ common
neighbors, arising from the two vertices $(i',j')$ with
$i'\notin\{i,k\}$ and $j'\ne j$; the same holds for $c_{\ell}$. Hence
$J(a,b)$ is $6w$-regular and every edge lies in exactly
$2w=\frac n6$ triangles, so $J(a,b)$ is tight, with $t(J(a,b))=24w^3$. 
Identifying $r_i$ with $(i,4)$
shows that $J(w,0)$ is the balanced blow-up $H_{4,w}$ of
$K_3\otimes K_4$ from Example \ref{ex:EFR}. Transposing
$[3]\times[3]$ shows that $J(a,b)\cong J(b,a)$. For
$0\le b\le a$ with $a+b=w$, the graphs $J(a,b)$ are pairwise
non-isomorphic.

The following result shows that the number of tight graphs, that is, graphs attaining
equality in Theorem \ref{thm-confirm-Zhai-Lin} at the threshold
$\lambda(G)=\lambda(T_{n,2})$, grows at least linearly with $n$. 

\begin{proposition}\label{prop:many}
For every $n\equiv 0\pmod{12}$, there are at least $\frac n{24}+2$
pairwise non-isomorphic tight graphs on $n$ vertices, and for
every $n\equiv 6\pmod{12}$ there is at least one. \end{proposition}

\begin{proof}
Let $n=12w$. The graphs $Y_{2w,2w}$, $\overline I_{12}(w)$ and
$J(a,b)$ with $0\le b\le a$ and $a+b=w$ are tight and have
$16w^3$, $20w^3$ and $24w^3$ triangles respectively, so no graph
of one family is isomorphic to a graph of another. By Example
\ref{ex:J}, the $\lfloor w/2\rfloor+1$ graphs $J(a,b)$ are
pairwise non-isomorphic, and
$\lfloor w/2\rfloor+3\ge\frac w2+2=\frac n{24}+2$. For $n=6k$
with $k$ odd, $Y_{k,k}$ is tight.
\end{proof}

\subsection{Unbalanced blow-ups and the strict inequality}
\label{sec:unbalanced}

We now perturb tight graphs so that the edge count exceeds
$e(T_{n,2})$ by exactly one.

\begin{proposition}\label{prop:triangle-lift}
Let $F$ be a tight graph on $N$ vertices, let $T$ be a triangle
of $F$, and let $t\ge1$. Let $G$ be the blow-up of $F$ in which
the three vertices of $T$ are replaced by independent sets of
size $t+1$ and every other vertex by an independent set of size
$t$. Then $n:=|V(G)|=Nt+3$ and
\[
e(G)=\Bigl\lfloor\frac{n^2}{4}\Bigr\rfloor+1,
\quad
\bk(G)=\Bigl\lfloor\frac n6\Bigr\rfloor+1,
\quad
\lambda(G)>\lambda(T_{n,2}).
\]
\end{proposition}

\begin{proof}
Put $t_v=t+1$ for $v\in T$ and $t_v=t$ otherwise, so
$n=Nt+3$. Since $F$ is $\frac N2$-regular with
$e(F)=N^2/4$ and $e(F[T])=3$,
\[
e(G)=\sum_{uv\in E(F)}t_ut_v
=\frac{N^2}{4}t^2+t\sum_{v\in T}d_F(v)+3
=\frac{N^2}{4}t^2+\frac{3N}{2}t+3 .
\]
Here $N$ is even, so $n$ is odd and
$\lfloor n^2/4\rfloor=\frac{N^2}{4}t^2
+\frac{3N}{2}t+2$, which gives the first equality. The third
assertion follows, since
$\lambda(G)\ge\frac{2e(G)}{n}
=\frac{2}{n}\bigl(\lfloor n^2/4\rfloor+1\bigr)
>\frac n2\ge\lambda(T_{n,2})$.

For the second, let $uv\in E(F)$ and put
$D=N_F(u)\cap N_F(v)$. The codegree in $G$ of an edge
joining the corresponding classes is $t|D|+|D\cap T|$.
If $|D\cap T|\ge2$, then two adjacent vertices of $T$
together with $u$ and $v$ span a $K_4$ in $F$,
contradicting Lemma \ref{lem:structure}(a); hence
$|D\cap T|\le1$, and $|D|\le N/6$ gives $\bk(G)\le tN/6+1$.
Conversely write $T=\{a,b,c\}$. Then $a\in N_F(b)\cap N_F(c)$,
so the edge $bc$ has positive codegree and therefore codegree
exactly $N/6$, and the edges of $G$ joining the classes of $b$
and $c$ lie in $tN/6+1$ triangles. 
Since $6\mid N$, we have 
$\lfloor n/6\rfloor=\lfloor (Nt+3)/6\rfloor=tN/6$.
\end{proof}

Proposition \ref{prop:triangle-lift} applies to every tight
graph. Applied to $Y_{k,k}$ it gives graphs on every
$n\equiv 3\pmod 6$ vertices; applied to $H_{4,t}$,
$\overline I_{12}(w)$ and $J(a,b)$ it gives further families on
$n\equiv 3\pmod{12}$ vertices. Hence, Theorem
\ref{thm-confirm-Zhai-Lin} cannot be improved, and the constant
$\frac16$ is best possible.

\subsection{Tight graphs of small order}\label{sec:small}

Definition \ref{def:tight} makes the search for tight graphs a finite problem, and Lemmas \ref{lem:structure}
and \ref{lem:parity} reduce it for an exhaustive computer
search. We list some tight graphs on at most $42$ vertices in Table
\ref{tab:tight-small}.

\begin{table}[htbp]
\centering
\begin{tabular}{@{}ccl@{}}
\toprule
$n$ & Number & Tight graphs on $n$ vertices \\
\midrule
$6$  & $1$ & $Y_{1,1}=K_3\,\square\, K_2$ \\
$12$ & $3$ & $Y_{2,2}$, \ $\overline I_{12}$, \
              $J(1,0)=K_3\otimes K_4$ \\
$18$ & $1$ & $Y_{3,3}$ \\
$24$ & $4$ & $Y_{4,4}$, \ $\overline I_{12}(2)$, \ $J(2,0)$, \
              $J(1,1)$ \\
$30$ & $1$ & $Y_{5,5}$ \\
$36$ & $4$ & $Y_{6,6}$, \ $\overline I_{12}(3)$, \ $J(3,0)$, \
              $J(2,1)$ \\
$42$ & $1$ & $Y_{7,7}$ \\
\bottomrule
\end{tabular}
\caption{Some tight graphs on at most $42$ vertices.}
\label{tab:tight-small}
\end{table}

\section{Stability results}

\label{sec:stability}

In this section, we prove Theorems \ref{thm:edit-distance}
and \ref{thm:edge-spectral-stability}. Both theorems have the
same shape: a graph whose spectral radius is close to the
threshold either contains a book of almost extremal size or
is close to the extremal graph in edit distance. Both proofs rely on the machinery of Section \ref{sec:weighted}. The
comparison of $m$ with $\lambda^2$ separates the two cases, and \eqref{eq:TS},
\eqref{eq:rhoT} and \eqref{eq:T-upper} supply the estimates
in the case $m\le\lambda^2$.
The two proofs differ in how they locate the near-extremal
graph. 

For a prescribed order, the case $m>\lambda^2$ is
covered by the stability theorem of Bollob\'as and Nikiforov,
and in the case $m\le\lambda^2$ Theorem
\ref{thm:edge-spectral-intro} already forces a large book
unless $G$ is complete bipartite. This yields Theorem
\ref{thm:direct}, a spectral form of the Bollob\'as--Nikiforov
theorem with the same explicit exponents, from which Theorem
\ref{thm:edit-distance} follows by balancing the two sides.

For a prescribed size, there is no edge condition to fall
back on, because the hypothesis
$\lambda(G)\ge(1-\delta)\sqrt m$ has no counterpart in terms
of $m$ alone. Instead, Lemma \ref{lem:weighted-cut} uses the
weighted quantities $T$ and $S$ to find a vertex whose
neighborhood cut has few edges on either side, and Lemma
\ref{lem:bipartite-completion} shows that a bipartite graph
with small spectral deficit $m-\lambda^2$ is close to a
complete bipartite graph, with edit distance linear in the
deficit. In both theorems, the dependence of $\delta$ on
$\varepsilon$ is explicit.

\subsection{Proof of Theorem \ref{thm:edit-distance}}

We start from the stability theorem of Bollob\'as and
Nikiforov \cite{BN2005} for the edge condition. 

\begin{theorem}[Bollob\'{a}s--Nikiforov \cite{BN2005}]\label{thm:BN-stability}
Let $0<\alpha<10^{-5}$. If $G$ is an $n$-vertex graph satisfying
$$e(G)\ge \left( \frac{1}{4} -\alpha \right)n^2,$$
 then either $G$ contains a book of size 
$\bk(G)>\left(\frac16-2\alpha^{1/3}\right)n$, 
or $G$ contains an induced bipartite subgraph $H$ of order  $ |H|\ge \left(1-\alpha^{1/3} \right)n$ and with minimum degree $ \delta(H)\ge \left(\frac12-4\alpha^{1/3}\right)n$. 
\end{theorem}

Theorem \ref{thm:edit-distance} follows from the following spectral counterpart of Theorem \ref{thm:BN-stability}. 

\begin{theorem}\label{thm:direct}
Let $0<\alpha<10^{-5}$ and let $G$ be an $n$-vertex graph satisfying
\begin{equation*} 
\lambda(G) \ge \left(\frac12 -\alpha\right)n.
\end{equation*}
Then one of the following holds:
{\rm (a)} $G$ contains a book of size 
$\bk(G)>\left(\frac16-2\alpha^{1/3}\right)n$; 
{\rm (b)} $G$ contains an induced bipartite subgraph $H$ such that 
$ |H|\ge \left(1-\alpha^{1/3} \right)n$ and 
$ \delta(H)\ge \left(\frac12-4\alpha^{1/3}\right)n$.
\end{theorem}

The standard route to a statement of this type is
Nikiforov's reduction \cite{Niki2009ejc,Niki2008}: one
passes to an induced subgraph $H$ on $(1-o(1))n$
vertices with either $\lambda(H)>(\frac12+o(1))|H|$, or
$\lambda(H)\ge(\frac12-o(1))|H|$ and
$\delta(H)\ge(\frac12-o(1))|H|$, and then treats the two
cases separately, using Theorem \ref{thm-confirm-Zhai-Lin} in the
first and the Andr\'asfai--Erd\H{o}s--S\'os theorem
\cite{AES1974} in the second. Our proof of Theorem
\ref{thm:direct} avoids this reduction; it is shorter and
gives explicit exponents. The comparison of $m$ with
$\lambda^2$ separates the two alternatives: Theorem
\ref{thm:BN-stability} handles the case $m > \lambda^2$, and
Theorem \ref{thm:edge-spectral-intro} settles the case
$m\le\lambda^2$.

\begin{proof}[Proof of Theorem \ref{thm:direct}]
Write $m=e(G)$ and $\lambda=\lambda(G)$. The
hypothesis  gives
\begin{equation}\label{eq:direct-rho-square}
\lambda^2\ge\Bigl(\frac14-\alpha\Bigr)n^2 .
\end{equation}
 We distinguish two cases according to the
position of $m$ relative to $\lambda^2$.

\smallskip
\noindent\textbf{Case 1: $m>\lambda^2$.}
Together with \eqref{eq:direct-rho-square}, this gives
$e(G)>\bigl(\frac14-\alpha\bigr)n^2$, and Theorem
\ref{thm:BN-stability} applied to $G$ yields alternative
(a) or alternative (b).

\smallskip
\noindent\textbf{Case 2: $m\le\lambda^2$.} 
Let $H$ be obtained from $G$ by deleting all  isolated vertices. Then $H$ is an induced subgraph of $G$ with $e(H)=m$ and $\lambda (H)=\lambda $. Consequently, we have $\lambda^2 (H) \ge (\frac{1}{4} - \alpha)n^2$ and $\lambda (H)\ge \sqrt{m}$. 
Assume first that $H$ is not complete bipartite. 
Then Theorem \ref{thm:edge-spectral-intro} applies to $H$ and
gives $\bk(H)\ge\lambda(H)-\frac{2m}{3\lambda(H)}
\ge\frac13\lambda(H)$. 
Combining with \eqref{eq:direct-rho-square}, we get
\[
\bk(G) = \bk (H)\ge \frac{1}{3} \lambda(H) 
\ge\frac n6\sqrt{1-4\alpha}
\ge\Bigl(\frac16-\frac{2\alpha}{3}\Bigr)n .
\]
Since $\frac{2\alpha}{3}<2\alpha^{1/3}$, alternative (a)
holds.

Assume now that $H=K_{s,t}$ for some integers
$1\le s\le t$ and $s+t\le n$. Then $\lambda(H) 
=\sqrt{st}$, so
\eqref{eq:direct-rho-square} becomes
$st\ge\bigl(\frac14-\alpha\bigr)n^2 $.
By the AM-GM inequality, we have 
\[
|H|=s+t\ge 2\sqrt{st}\ge n\sqrt{1-4\alpha}
\ge(1-4\alpha)n>\bigl(1-\alpha^{1/3}\bigr)n .
\]
For the minimum degree, note that $t\le n-s$, so
$s(n-s)\ge st\ge 
\bigl(\frac14-\alpha\bigr)n^2$; that is,
$(\frac sn-\frac{1}{2})^2\le\alpha$. Hence
\[
\delta(H)=s\ge\Bigl(\frac12-\sqrt\alpha\Bigr)n
>\Bigl(\frac12-4\alpha^{1/3}\Bigr)n .
\]
Thus $H$ is an induced bipartite subgraph 
satisfying alternative (b).
\end{proof}

Now, we are ready to prove Theorem \ref{thm:edit-distance}. 

\begin{proof}[{\bf Proof of Theorem \ref{thm:edit-distance}}]
Put $\delta=\min\{10^{-6},(\varepsilon/17)^3\}$ and
$\theta=\delta^{1/3}$, so that $\theta\le\varepsilon/17$
and $\theta\le10^{-2}$. Apply Theorem \ref{thm:direct}
with $\alpha=\delta$. If alternative (a) holds, then
\[
\bk(G)>\Bigl(\frac16-2\theta\Bigr)n
\ge\Bigl(\frac16-\varepsilon\Bigr)n. 
\]
 Assume that alternative (b)
holds, and fix an induced bipartite subgraph $H$ of $G$
 such that $|H|\ge(1-\theta)n$ and
$\delta(H)\ge\bigl(\frac12-4\theta\bigr)n$. 
Put $V(H)=X\sqcup Y$ and 
$R=V(G)\setminus V(H)$, so that $|R|\le\theta n$. 

Every $H$-neighbor of a vertex of $X$ lies in $Y$,
whence $|Y|\ge\delta(H)\ge(\frac12-4\theta)n$; by
symmetry the same bound holds for $|X|$. Since
$|X|+|Y|\le n$, we obtain
\begin{equation}\label{eq:parts-balanced}
\Bigl(\frac12-4\theta\Bigr)n\le|X|,|Y|
\le\Bigl(\frac12+4\theta\Bigr)n .
\end{equation}
Hence each $x\in X$ is nonadjacent to at most
$|Y|-\deg_H(x)\le8\theta n$ vertices of $Y$, and so
\begin{equation}\label{eq:missing-cross}
\bigl|\{\,xy:x\in X,\ y\in Y,\ xy\notin E(G)\,\}\bigr|
\le 8\theta n^2 .
\end{equation}
Moreover $H$ is induced and bipartite, so $G$ has no
edge inside $X$ and no edge inside $Y$.

We next balance the two sides. Put $p=\lfloor n/2\rfloor$.
We choose a set $X'$ with $|X'|=p$ as follows, and we put
$Y'=V(G)\setminus X'$, so that $|Y'|=\lceil n/2\rceil$. If
$|X|\ge p$, choose $X'\subseteq X$. Then $Y\subseteq Y'$,
and \eqref{eq:parts-balanced} together with $p\ge(n-1)/2$
gives $|X\setminus X'|=|X|-p\le 4\theta n+\frac12$. If
$4\theta n\ge\frac12$, this is at most $8\theta n$; if
$4\theta n<\frac12$, then $|X|-p$ is a nonnegative integer
smaller than $1$, hence zero. In either case
$|X\setminus X'|\le 8\theta n$. If $|X|<p$, choose
$X'\supseteq X$ by adding $p-|X|\le 4\theta n$ vertices of
$R\cup Y$. Then $X\subseteq X'$ and
$|Y\setminus Y'|=|Y\cap X'|\le 4\theta n$. In both cases, 
we put 
\[
B=R\cup(X\setminus X')\cup(Y\setminus Y') ,
\qquad
|B|\le\theta n+8\theta n=9\theta n .
\]

Let $K$ be the complete bipartite graph with parts $X'$
and $Y'$; since $|X'|=\lfloor n/2\rfloor$ and
$|Y'|=\lceil n/2\rceil$, we have $K\cong T_{n,2}$. We
estimate $|E(G)\bigtriangleup E(K)|$ by splitting the
pairs of vertices according to whether they meet $B$.
The pairs meeting $B$ number at most
$|B|n\le9\theta n^2$. A pair avoiding $B$ has both ends
in $(X\cap X')\cup(Y\cap Y')$. If both ends lie in
$X\cap X'\subseteq X$, or both lie in
$Y\cap Y'\subseteq Y$, then the pair is a nonedge of $G$
and a nonedge of $K$, so it contributes nothing.
Otherwise one end lies in $X$ and the other in $Y$, the
pair is an edge of $K$, and by \eqref{eq:missing-cross}
at most $8\theta n^2$ such pairs fail to be edges of
$G$. Therefore
$|E(G)\bigtriangleup E(K)|
\le 9\theta n^2+8\theta n^2=17\theta n^2
\le\varepsilon n^2 $,
which is the second alternative.
\end{proof}

\subsection{Proof of Theorem \ref{thm:edge-spectral-stability}}

The argument has two independent ingredients. First,
the Perron-weighted quantities $T$ and $S$ of Section
\ref{sec:weighted} identify a cut with few internal
edges. Second, a bipartite graph that is nearly extremal
for the inequality $\lambda(F)^2\le e(F)$ is close in
edit distance to a complete bipartite graph.

The second ingredient is the following lemma. Li, Liu
and Zhang \cite[Lemma 4.4]{LLZ-edge-spectral} proved by a Perron vector analysis that for
$\varepsilon\in(0,0.01)$ every bipartite graph with
$\lambda(G)\ge(1-\varepsilon^4/100)\sqrt m$ is within edit
distance $\varepsilon m$ of a complete bipartite graph.
Our lemma bounds the edit distance by an explicit linear
function of the spectral deficit $m-\lambda^2$, so the
threshold $\varepsilon^4/100$ improves to one that is
linear in $\varepsilon$, and the proof runs through
$2\times2$ minors instead of the Perron vector.

\begin{lemma} \label{lem:bipartite-completion}
Let $G$ be a bipartite graph with $m $ edges and spectral radius $\lambda=\lambda(G)$. Then there exist disjoint vertex sets $X,Y\subseteq V(G)$ such that
\begin{equation*}
\bigl| E(G)\, {\triangle}\, E(K_{X,Y})\bigr|
\le 3\bigl( m -\lambda^2\bigr).
\end{equation*}
Moreover, $X$ and $Y$ can be taken to be
$N(v)$ and $N(u)$ for a suitable edge $uv$ of $G$.
\end{lemma}

\begin{proof}    
If $m=0$, take $X=Y=\varnothing$. Assume $m\ge1$.
Fix a bipartition $V(G)=U\sqcup W$, and let $M$ be the
corresponding biadjacency matrix: a $0$--$1$ matrix whose rows
are indexed by $U$ and whose columns are indexed by $W$, with
$M_{uw}=1$ if and only if $uw\in E(G)$. Then $\lambda$ is the
largest singular value of $M$. Write $r=\operatorname{rank} M$
and 
$\sigma_1=\lambda\ge\sigma_2\ge\cdots \ge \sigma_r > 0$
for the positive singular values of $M$. Since $M$ has exactly $m$ entries equal to $1$, we have 
\begin{equation*}
\sum_i\sigma_i^2=\lVert M\rVert_F^2=m.
\end{equation*}
Put
 $\xi :=m-\lambda^2=\sum_{i\ge2}\sigma_i^2$. 
In particular, we have $\xi\ge0$. 
Let
\begin{equation*}
\mathcal D:=
\sum_{\substack{I\subseteq U, |I|=2\\
J\subseteq W, |J|=2}}
\det(M[I,J])^2
\end{equation*}
be the sum of the squared $2\times2$ minors of $M$. By the
Cauchy--Binet formula applied to $MM^{\mathsf T}$, for every
$I\subseteq U$ with $|I|=2$, we have
$$ \det\bigl((MM^{\mathsf T})[I,I]\bigr)
=\sum_{J\subseteq W, |J|=2}\det(M[I,J])^2.$$
Summing over all such $I$ expresses $\mathcal D$ as the sum
of the principal $2\times2$ minors of $MM^{\mathsf T}$,
which is the second elementary symmetric function of its
eigenvalues. The nonzero eigenvalues of $MM^{\mathsf T}$
are $\sigma_1^2,\ldots,\sigma_r^2$, so
\begin{equation*}
\mathcal D=\sum_{i<j}\sigma_i^2\sigma_j^2.
\end{equation*}
By the definition of $\xi$, we obtain
\begin{align}
\mathcal D
=\lambda^2\xi+
  \sum_{2\le i<j}\sigma_i^2\sigma_j^2 
  \le \lambda^2\xi+\frac{\xi^2}{2} 
  =(m-\xi)\xi+\frac{\xi^2}{2} 
  \le m \,\xi.\label{eq:minor-defect-upper}
\end{align}
For an edge $uv\in E(G)$ with $u\in U$ and $v\in W$, we define
\[
B_{uv}:=N(v)\times N(u).
\]
Thus $B_{uv}$ is the edge set of the complete bipartite graph $K_{N_G(v),N_G(u)}$. Set
\[
D_{uv}:=\bigl|E(G)\, {\triangle}\, B_{uv}\bigr|.
\]
We now estimate the average of $D_{uv}$ over all edges $uv$. 

Because $M$ is a $(0,1)$-matrix, a $2\times 2$ minor of $M$ is nonzero precisely when one diagonal consists of two entries equal to $1$, whereas the product of the two entries on the other diagonal is $0$. Suppose that $uv,u'v'\in E(G)$ and that $u'v'\notin B_{uv}$. Then at least one of $uv'$ and $u'v$ is absent. 
In particular $u'\ne u$ and $v'\ne v$, because $u'=u$ would give $uv'=u'v'\in E(G)$ and $v'=v$ would give $u'v=u'v'\in E(G)$. Hence the submatrix on rows $u,u'$ and columns $v,v'$ has nonzero determinant. Conversely, every nonzero $2\times2$ minor has a unique diagonal consisting of two present edges, and each orientation of that diagonal gives one ordered pair $(uv,u'v')$ with $u'v'\notin B_{uv}$. Consequently, we obtain 
\begin{equation}\label{eq:edges-outside-rectangles}
\sum_{uv\in E(G)}|E(G)\setminus B_{uv}|=2\mathcal D.
\end{equation}

Next suppose that $u'v'\in B_{uv}\setminus E(G)$. By the definition of $B_{uv}$, we have 
\[
uv,\quad uv',\quad u'v\in E(G),
\qquad
u'v'\notin E(G).
\]
Since $uv'\in E(G)$ and $u'v'\notin E(G)$, we have $u'\ne u$; likewise $v'\ne v$. Hence the corresponding $2\times2$ submatrix of $M$ has exactly three entries equal to $1$ and has nonzero determinant. Conversely, each $2\times2$ submatrix with exactly three entries equal to $1$ determines exactly one pair consisting of the edge opposite the missing entry and that missing entry. Therefore
\begin{equation}\label{eq:missing-inside-rectangles}
\sum_{uv\in E(G)}|B_{uv}\setminus E(G)|\le\mathcal D.
\end{equation}
Combining \eqref{eq:edges-outside-rectangles} and \eqref{eq:missing-inside-rectangles}, we find that
\[
\sum_{uv\in E(G)}D_{uv}\le3\mathcal D.
\]
Averaging over the $m$ edges, some edge $uv\in E(G)$
satisfies $D_{uv}\le\frac{3\mathcal D}{m}$.
By \eqref{eq:minor-defect-upper},
we have $\mathcal{D} \le m\,\xi$ and  
$D_{uv}\le3\xi=3(m-\lambda^2)$.
Taking $X=N(v)$ and $Y=N(u)$, we obtain the desired complete bipartite graph $K_{X,Y}$.
\end{proof}

\begin{lemma}\label{lem:weighted-cut}
Let $H$ be a connected graph with $h\ge1$ edges,
spectral radius $\lambda$, and booksize $\beta$. If
$\lambda>3\beta$, then $h\ge\lambda^2$ and there exists
a vertex $v\in V(H)$ such that
\begin{equation}\label{eq:weighted-cut-bound}
\begin{aligned}
e\bigl(H[N(v)]\bigr)
+e\bigl(H[V(H)\setminus N(v)]\bigr)
 \le
(h-\lambda^2)
\left(1+\frac{2\beta}{\lambda-3\beta}\right).
\end{aligned}
\end{equation}
\end{lemma}

\begin{proof}
Apply the notation of Subsection \ref{sec:three-ineq}  to $H$: let $\bm{x}$ be a Perron vector of $H$ with
$\|\bm{x}\|_1=1$, and let $t_v$, $s_v$, $T$, $S$ be the associated quantities, so that \eqref{eq:rhoT} reads $\lambda T\le\beta(2T+S)$. Note that no edge incident
with $v$ has both ends outside $N(v)$, so
\begin{equation}\label{eq:sv-open-closed}
e\bigl(H[V(H)\setminus N(v)]\bigr)
=e\bigl(H[V(H)\setminus N_H[v]]\bigr)=s_v ,
\end{equation}
and the left-hand side of
\eqref{eq:weighted-cut-bound} is exactly $t_v+s_v$.

Put $\theta:=\lambda^2-h$. Substituting $S=T-\theta$
from \eqref{eq:TS} into \eqref{eq:rhoT} gives
$\lambda T\le\beta(3T-\theta)$, that is,
\begin{equation}\label{eq:weighted-cut-T}
(\lambda-3\beta)\,T\ \le\ \beta\,(h-\lambda^2).
\end{equation}

We first check that $h\ge\lambda^2$. If $\beta>0$, this
follows from \eqref{eq:weighted-cut-T}, because
$\lambda-3\beta>0$ and $T\ge0$ force the right-hand
side to be nonnegative. If $\beta=0$, then $H$ is
triangle-free, so $t_v=0$ for every $v$ and $T=0$;
now \eqref{eq:TS} gives $S=h-\lambda^2$, which is
nonnegative because every $s_v$ is. In either case
$h\ge\lambda^2$, and \eqref{eq:weighted-cut-T} yields
\begin{equation}\label{eq:weighted-cut-T-upper}
T\ \le\ \frac{\beta(h-\lambda^2)}{\lambda-3\beta}. 
\end{equation}
By the definitions of $T$ and $S$ together with
\eqref{eq:sv-open-closed},
\[
\sum_{v\in V(H)}x_v\bigl(t_v+s_v\bigr)=T+S .
\]
Using $S=T-\theta=T+(h-\lambda^2)$ and
\eqref{eq:weighted-cut-T-upper},
\[
T+S=(h-\lambda^2)+2T
\le(h-\lambda^2)
\left(1+\frac{2\beta}{\lambda-3\beta}\right).
\]
Since $x_v>0$ for every $v$ and $\sum_v x_v=1$, the
left-hand side is a weighted average of the quantities
$t_v+s_v$, so at least one vertex $v$ satisfies
\eqref{eq:weighted-cut-bound}.
\end{proof}

Now, we are ready to prove Theorem \ref{thm:edge-spectral-stability}. 

\begin{proof}[{\bf Proof of Theorem \ref{thm:edge-spectral-stability}}]
Put $\gamma=\min \{\varepsilon,\frac16 \},
\delta=\frac{\gamma^{3}}{1000}$ and  
$\eta
=\frac{\gamma^{2}}{1000}$.
 Write $\lambda=\lambda(G)$ and $\beta=\bk(G)$, and assume
that $\lambda\ge(1-\delta)\sqrt m$. Suppose that the first
alternative fails, so that
$\beta\le(\tfrac13-\varepsilon)\sqrt m
\le(\tfrac13-\gamma)\sqrt m$; we show that the second
alternative holds.
Choose a connected component $H$ of $G$ with
$\lambda(H)=\lambda$, and put $h=e(H)$ and
$\beta_H=\bk(H)$. 
As $\beta_H\le\beta$, the two hypotheses  yield
\[
\lambda-3\beta_H\ \ge\ \lambda-3\beta
\ \ge\ (1-\delta)\sqrt m-(1-3\gamma)\sqrt m
\ =\ (3\gamma-\delta)\sqrt m
\ >\ 2\gamma\sqrt m ,
\]
where the last step holds because
$\delta=\gamma^{3}/1000<\gamma$. In
particular $\lambda>3\beta_H$, so Lemma
\ref{lem:weighted-cut} applies to $H$ and gives
$h\ge\lambda^{2}$. Combined with $h\le m$,  this yields
$\lambda\le\sqrt m$ and
\[
0\ \le\ m-\lambda^{2}
\ \le\ \bigl(1-(1-\delta)^{2}\bigr)m
\ <\ 2\delta m .
\]
Moreover, 
Lemma \ref{lem:weighted-cut} also provides a vertex
$v\in V(H)$ with
\[
q_H(v):=e\bigl(H[N_H(v)]\bigr)
+e\bigl(H[V(H)\setminus N_H(v)]\bigr)
\ \le\ (h-\lambda^{2})
\Bigl(1+\frac{2\beta_H}{\lambda-3\beta_H}\Bigr).
\]
Set $A_0=N_H(v)=N_G(v)$ and $B_0=V(G)\setminus A_0$,
and let $F=G[A_0,B_0]$ be the spanning bipartite
subgraph formed by the edges of $G$ crossing the cut
$A_0\sqcup B_0$. The edges of $G$ missing from $F$ are
those inside $A_0$ and those inside $B_0$. Since
$A_0\subseteq V(H)$, the former are the
$e\bigl(H[N_H(v)]\bigr)$ edges inside $N_H(v)$. Since
$V(G)\setminus V(H)\subseteq B_0$ and $G$ has no edge
joining $V(H)$ to $V(G)\setminus V(H)$, the latter
consist of the $e\bigl(H[V(H)\setminus N_H(v)]\bigr)$
edges inside $B_0\cap V(H)$ together with the $m-h$
edges outside $H$. Hence
\[
q:=\bigl|E(G)\setminus E(F)\bigr|=(m-h)+q_H(v).
\]

We now bound $q$. From $\lambda-3\beta>0$ we get
$\beta<\lambda/3$, and $t\mapsto 2t/(\lambda-3t)$ is
increasing on $[0,\lambda/3)$; so, using
$\beta_H\le\beta<\frac13\sqrt m$ and
$\lambda-3\beta>2\gamma\sqrt m$,
\[
1+\frac{2\beta_H}{\lambda-3\beta_H}
\ \le\ 1+\frac{2\beta}{\lambda-3\beta}
\ <\ 1+\frac{\tfrac23\sqrt m}{2\gamma\sqrt m}
\ =\ 1+\frac{1}{3\gamma}
\ \le\ \frac{1}{2\gamma},
\]
the final inequality being equivalent to
$6\gamma+2\le3$, which holds because
$\gamma\le\frac16$. This factor is at least
$1$, and $(m-h)+(h-\lambda^{2})=m-\lambda^{2}$, so
\[
q\ \le\ (m-h)+(h-\lambda^{2})
\Bigl(1+\frac{2\beta_H}{\lambda-3\beta_H}\Bigr)
\ \le\ (m-\lambda^{2})
\Bigl(1+\frac{2\beta_H}{\lambda-3\beta_H}\Bigr)
\ <\ 2\delta m\cdot\frac{1}{2\gamma}
\ =\ \eta m .
\]

Put $\ell=e(F)=m-q$ and $\lambda_F=\lambda(F)$. As
$\eta<1$, we have $q<m$, so $\ell\ge1$ and $F$ has an edge. 
The matrix $A(G)-A(F)$ is symmetric with exactly $2q$ nonzero entries, each equal to $1$, so $\lVert A(G)-A(F)\rVert_F=\sqrt{2q}$. Weyl's inequality and
$\lVert\cdot\rVert_2\le\lVert\cdot\rVert_F$  give
\[
\lambda_F\ \ge\ \lambda-\lVert A(G)-A(F)\rVert_2
\ \ge\ \lambda-\sqrt{2q}
\ >\ \bigl(1-\delta-\sqrt{2\eta}\bigr)\sqrt m ,
\]
the last step by $q<\eta m$. Moreover
$\gamma\le\frac16$ gives
$\delta+\sqrt{2\eta}
 < 1 $, 
so $\lambda-\sqrt{2q}>0$ and the bound
$\lambda_F\ge\lambda-\sqrt{2q}$ may be squared. Using
$\lambda\le\sqrt m$, then $m-\lambda^{2}<2\delta m$ and
$q<\eta m$,
\[
\ell-\lambda_F^{2}
\ \le\ (m-q)-\bigl(\lambda-\sqrt{2q}\bigr)^{2}
\ \le\ (m-\lambda^{2})+2\sqrt{2mq}
\ <\ \bigl(2\delta+2\sqrt{2\eta}\bigr)m .
\]

Applying Lemma \ref{lem:bipartite-completion} to the
bipartite graph obtained from $F$ by deleting its isolated
vertices, which has $\ell$ edges and spectral radius
$\lambda_F$, we obtain disjoint sets $A,B\subseteq V(G)$ with
\[
\bigl|E(F)\,\triangle\,E(K_{A,B})\bigr|
\ \le\ 3\bigl(\ell-\lambda_F^{2}\bigr)
\ <\ \bigl(6\delta+6\sqrt{2\eta}\bigr)m .
\]
Since $E(G)\,\triangle\,E(K_{A,B})$ is contained in
$\bigl(E(G)\,\triangle\,E(F)\bigr)\cup
\bigl(E(F)\,\triangle\,E(K_{A,B})\bigr)$ and
$|E(G)\,\triangle\,E(F)|=q<\eta m$, we get
\[
\bigl|E(G)\,\triangle\,E(K_{A,B})\bigr|
<\bigl(\eta+6\delta+6\sqrt{2\eta}\bigr)m .
\]
Finally, $\gamma\le\frac16$ gives
$\eta+6\delta=\frac{\gamma^{2}}{1000}+\frac{6\gamma^{3}}{1000}
<\frac{\gamma}{1000}$ and
$6\sqrt{2\eta}=6\gamma\sqrt{\tfrac{2}{1000}}<\frac{3\gamma}{10}$,
so the right-hand side is less than
$\gamma m\le\varepsilon m$. This is the edit-distance
alternative, and the proof is complete.
\end{proof}

\section{Concluding remarks}

\label{sec:conclusion}

This paper determines sharp bounds in the spectral book problem.  
In the vertex-spectral setting, we proved
that $\bk(G)\ge\frac13\lambda(G)$ whenever
$\lambda(G)\ge\lambda(T_{n,2})$ and $G\neq T_{n,2}$ (Theorem \ref{thm-confirm-Zhai-Lin}), 
which solves Problem
\ref{prob-ZL} of Zhai and Lin \cite{ZhaiLin2023} in a stronger form.   
The extremal graphs of Theorem \ref{thm-confirm-Zhai-Lin} are
exactly the $\frac n2$-regular graphs in which every edge lies in $0$ or $\frac n6$ triangles (Proposition \ref{prop:equality}), and there are infinitely many extremal graphs. 
We refined the Edwards bound to
$\bk(G)\ge\max\{\lambda-\frac n3,\,2\lambda-n\}$ (Theorem \ref{thm:fixed-order-intro}). 
In the edge-spectral setting,  every $m$-edge  Nosal graph satisfies $\bk(G)\ge\lambda-\frac{2m}{3\lambda}$ (Theorem \ref{thm:edge-spectral-intro}), 
which is attained  by infinitely many extremal graphs.  The companion bound $\bk(G)\ge 2\lambda-\frac{2m}{\lambda}$ is attained by every regular complete multipartite graph (Theorem
\ref{thm:two-lambda}). 
As an application, we improved the  Bollob\'as--Nikiforov bound to $t(G)\ge\frac13(\lambda+1)(\lambda^{2}-m)$ (Theorem \ref{thm:BN-refined}). 
Finally, we established stability results at both thresholds: a graph whose spectral radius is close to either threshold either contains a book of almost extremal size, or lies within
a small edit distance of the corresponding complete bipartite graph (Theorems \ref{thm:edit-distance} and \ref{thm:edge-spectral-stability}). 

All the results come from a single mechanism. 
For a connected graph $G$ with Perron vector $\bm{x}$, the two Perron-weighted parameters  $T=\sum_{v}x_{v}t_{v}$ and $S=\sum_{v}x_{v}s_{v}$ obey one identity and two inequalities (see Lemmas \ref{lem:degree-neighborhood},
\ref{lem:charging} and \ref{lem:T-upper}), and the booksize enters only through the pointwise bound
$t_{uv}\le\bk(G)$. Eliminating $T$ and $S$ between these relations produces the sharp constants in both the vertex-spectral and the edge-spectral settings. In addition, the same
relations then yield the refined Bollob\'as--Nikiforov
inequality of Theorem \ref{thm:BN-refined} and the two stability theorems.

We close with two topics beyond the $n/6$ threshold: a
spectral form of the Conlon--Fox--Sudakov problem, which we
propose as an open question, and a spectral analogue of
Khad\v{z}iivanov's theorem at the Tur\'an threshold for larger
cliques, which we prove in Theorem \ref{thm:main-3}.

\subsection{A spectral
Conlon--Fox--Sudakov problem}

Mubayi \cite{Mubayi2012} initiated the study of the
trade-off between $t(G)$ and $\bk(G)$ in $n$-vertex
graphs with more than $\lfloor n^{2}/4\rfloor$ edges.
It was shown in \cite{Mubayi2012} that 
for every $n$-vertex graph $G$ with $e(G)=\lfloor n^2/4 \rfloor +1$, if $\frac{1}{4} < \beta < \frac{1}{2}$  and 
$\bk(G)< \beta n$, 
then $G$ has at least $\frac{1}{2}\beta (1-2\beta) n^2 - o(n^2)$ triangles;  
if $ \frac{1}{6} < \beta < \frac{1}{4}$ and $\bk(G)< \beta n$,  then $G$ has at least $\gamma n^3$ triangles for some $\gamma >0$. 
For the latter case, Conlon, Fox and Sudakov \cite{CFS2020}
conjectured that if $n/6\le b<n/4$, and $G$ is an $n$-vertex graph with $e(G)\ge \lfloor n^{2}/4\rfloor$, $G\neq T_{n,2}$, and $\bk (G)\le b$, then $G$ contains at
least $b^{2}(n-4b)$ triangles. They proved this for
$b=n/6$ and for $0.2495\,n\le b<n/4$; the
Bollob\'as--Nikiforov inequality \eqref{eq-BN} is one of
the tools in the first case. The conjectured extremal graph
$S_{b,n}$ is the blow-up of the triangular prism, 
where four of the six parts, corresponding to the vertices of two edges of the matching, are of size $b$, and the remaining two parts are of size $\lfloor(n-4 b) / 2\rfloor$ and $\lceil(n-4 b) / 2\rceil$; for $b=n/6$ it is the balanced blow-up $Y_{k,k}$ of Example \ref{ex:prism}.
Recently, Chen, Ma and Wang \cite{ChenMaWang2026} confirmed
the conjecture for all $n/6\le b\le(\frac16+\varepsilon)n$ and some absolute constant $\varepsilon>0$, 
with $S_{b,n}$ as the unique extremal
graph. It is natural to
ask for a spectral counterpart. Note that Theorem
\ref{thm-confirm-Zhai-Lin} yields
$\bk(G)\ge\frac13\lambda(T_{n,2})\ge\frac{n-1}{6}$ under
the spectral hypothesis, so the assumption $b\ge n/6$
below is essentially no restriction.

\begin{problem}\label{prob:spectral-CFS}
Let $n/6\le b<n/4$ be an integer, and let $G$ be an $n$-vertex graph
with $\lambda(G)\ge\lambda(T_{n,2})$, $G\neq T_{n,2}$
and $\bk(G)\le b$. Is it true that
$t(G)\ge(1-o(1))\,b^{2}(n-4b)$?
\end{problem}

\subsection{A spectral Khad\v{z}iivanov theorem}

The Tur\'{a}n theorem states that if $e(G) > e(T_{n,r})$, then $G$ contains a copy of $K_{r+1}$, so $\bk(G)\ge r-1$.  
Khad\v{z}iivanov \cite[Theorem 2]{Khadziivanov1991} proved a much stronger bound at the Tur\'an threshold: for every $r\ge 2$, if $e(G) \ge \frac{r-1}{2r}n^2$, then $\bk (G) \ge \frac{r-2}{r}n$, with equality if and only if $r\mid n$ and $G=T_{n,r}$. 
We prove a spectral form of this result. We first record an elementary lower bound
on $\lambda(T_{n,r})$. Write $n=qr+s$ with $0\le s<r$. The part sizes of $T_{n,r}$ are $q+1$, repeated $s$ times, and
$q$, repeated $r-s$ times. Consequently,
$
2 e(T_{n,r}) = n^2 - \left(s(q+1)^2+(r-s)q^2\right) = \frac{r-1}{r}n^2-\frac{s(r-s)}r$. 
It follows that
\[
\lambda(T_{n,r}) \geq \frac{2e(T_{n,r})}{n} = \frac{r-1}{r}n - \frac{s(r-s)}{rn}.
\]
Since $s(r-s)\leq r^2/4$, it follows that
\begin{equation}\label{eq:lambda-Turan}
\lambda(T_{n,r}) \geq \frac{r-1}{r} n - \frac{r}{4n}.
\end{equation}

\begin{theorem}\label{thm:main-3}
Let $r\geq 2$ and $G$ be a graph of order $n$. If
$n > \frac{r^2}{2\cdot \gcd(2,r)}$ and $\lambda(G)\geq\lambda(T_{n,r})$, then
$$\bk(G)\geq\frac{r-2}{r}\,n.$$
\end{theorem}

\begin{proof}
Suppose on the contrary that $\bk(G) < \frac{r-2}{r}\,n$.
Put $g:=\gcd(2,r)$ and $x:=\frac{r-2}{r}\,n$.
The rational number $x$ has denominator dividing
\[
\frac{r}{\gcd(r-2,r)} = \frac{r}{\gcd(2,r)} = \frac{r}{g}.
\]
Therefore, if $x$ is not an integer, its fractional part is at
least $g/r$. If $x$ is an integer, every integer strictly smaller
than $x$ is at most $x-1$, and $1\ge g/r$. Since $\bk(G)$ is an
integer, the assumption $\bk(G)<\frac{r-2}{r}\,n$ therefore gives
\[
  \bk(G)\leq\frac{r-2}{r}\,n - \frac{g}{r}.
\]
Since $\lambda(G)\ge\lambda(T_{n,r})>0$, we have $m\ge1$,
and $\lambda(G)\ge\frac{2m}{n}$. Theorem \ref{thm:two-lambda},
applied to $G$ with its isolated vertices deleted, gives
$\bk(G)\ge2\lambda(G)-\frac{2m}{\lambda(G)}\ge2\lambda(G)-n$.
It follows that 
\begin{equation}\label{eq:temp}
\lambda(G)\leq\frac{n+\bk(G)}2 \leq
\frac{1}{2} \left(n+\frac{r-2}{r}n-\frac{g}{r}\right)
=\frac{r-1}{r}n-\frac{g}{2r}.
\end{equation}
On the other hand, \eqref{eq:lambda-Turan} and $\frac{r}{4n}<\frac{g}{2r}$ yield
\begin{equation}\label{eq:lambda-Turan-lower}
\lambda(T_{n,r})\geq\frac{r-1}{r}n - \frac{r}{4n}
> \frac{r-1}{r}n - \frac{g}{2r}.
\end{equation}
Combining \eqref{eq:temp} and \eqref{eq:lambda-Turan-lower} yields
$\lambda(G)<\lambda(T_{n,r})$, a contradiction.
\end{proof}

\begin{remark}
For $r\mid n$ the bound of Theorem \ref{thm:main-3} is
attained by $T_{n,r}$, which is regular of degree
$\lambda(T_{n,r})=\frac{r-1}{r}n$ and satisfies
$\bk(T_{n,r})=\frac{r-2}{r}n$.
\end{remark}

\section*{Acknowledgments}

After the second and third authors had submitted their paper for publication, they learned that the first author had independently solved Problem \ref{prob-ZL}. We therefore decided to combine our results and add several further results. 
The authors would like to thank Prof. Ping Hu for valuable suggestions. 
At an exploratory stage the authors used language-model-based tools for brainstorming; all arguments and proofs in this paper were developed, written and verified by the authors.

\appendix

\section{Alternative proof of 
Bollob\'as--Nikiforov's inequality}

\label{sec:App}

In this section, we give an alternative proof of
\eqref{eq-BN}. The inequality appears as Theorem 1 in
Khad\v{z}iivanov \cite{Khadziivanov1991}, where it comes with
a characterization of equality, and it follows from
\cite[Theorem 1]{BN2005}. The proof below is shorter than both
and isolates the source of the constant $1/3$.
Throughout, $G$ is a graph with $n$
vertices and $m$ edges, $k_3:=t(G)$ is the number of
triangles of $G$, and $Q=\sum_{v\in V(G)}d(v)^2$. For an
edge $uv$ we write $c_{uv}:=|N(u)\cap N(v)|$ for the
number of triangles containing $uv$, so that
$c_{uv}\le\bk(G)$, and for a vertex $u$ we put
$\overline N(u):=V(G)\setminus N(u)$. For an edge $uv$ of
$G$, we write 
\[
  S_{uv}:=\bigl(N(u)\cap N(v)\bigr)
        \cup\bigl(\overline N(u)\cap\overline N(v)\bigr).
\]
So $S_{uv}$ is the set of vertices that are adjacent
to both ends of $uv$ or to neither of them; the
vertices outside $S_{uv}$ are those adjacent to
exactly one end. Since the two sets in the union are
disjoint,
\begin{equation}\label{eq:edw-Suv}
  |S_{uv}|=c_{uv}
    +\bigl|\overline N(u)\cap\overline N(v)\bigr|
   =n-d(u)-d(v)+2c_{uv},
\end{equation}
where the second equality is inclusion and
exclusion:
$$\bigl|\overline N(u)\cap\overline N(v)\bigr|
 =n-|N(u)\cup N(v)|=n-d(u)-d(v)+c_{uv}.$$

The next lemma is the heart of the proof and the source
of the constant $1/3$.

\begin{lemma}\label{lem:edw-pigeonhole}
Let $T$ be a triangle of $G$ and let $w$ be any
vertex of $G$. Then $w\in S_e$ for at least one of
the three edges $e$ of $T$. Consequently
\begin{equation}\label{eq:edw-local}
  \sum_{e\in E(T)}|S_e|\ \ge\ n .
\end{equation}
\end{lemma}

\begin{proof}
Write $T=xyz$. The vertex $w$ splits $V(G)$ into the
two sets $N(w)$ and $\overline N(w)$. The three
vertices $x$, $y$, $z$ lie in these two sets, so by
the pigeonhole principle two of them lie in the same
set; say they are $x$ and $y$. If both lie in
$N(w)$, then $w$ is adjacent to both $x$ and $y$, so
$w\in N(x)\cap N(y)\subseteq S_{xy}$. If both lie in
$\overline N(w)$, then $w$ is adjacent to neither,
so $w\in\overline N(x)\cap\overline N(y)
\subseteq S_{xy}$. In both cases $w\in S_{xy}$, and
$xy$ is an edge of $T$. 
For \eqref{eq:edw-local}, sum over all $w\in V(G)$. Each
of the $n$ vertices is counted at least once on the
left, so $\sum_{e\in E(T)}|S_e|\ge n$.
\end{proof}

\begin{theorem}\label{thm:edw-master-ineq}
For every graph $G$ with $n$ vertices and $m$ edges,
\begin{equation*} 
  \bk (G) \cdot (Q - mn) \le  \big(6\bk (G) - n\big) \cdot t(G). 
\end{equation*}
\end{theorem}

\begin{proof}
Write $\beta=\bk(G)$. 
Sum \eqref{eq:edw-local} over all triangles $T$ of
$G$:
\begin{equation}\label{eq:edw-sum1}
  n\,k_3\ \le\ \sum_{T}\sum_{e\in E(T)}|S_e| .
\end{equation}
On the right side, group the terms by the edge $e$.
An edge $e$ lies in exactly $c_e$ triangles, so it
contributes the term $|S_e|$ exactly $c_e$ times.
Hence
\begin{equation}\label{eq:edw-sum2}
  \sum_{T}\sum_{e\in E(T)}|S_e|
  =\sum_{e\in E(G)}c_e\,|S_e|
  \ \le\ \beta\sum_{e\in E(G)}|S_e| ,
\end{equation}
where the inequality uses $c_e\le\beta$ for every
edge and $|S_e|\ge0$.

It remains to compute the last sum. By
\eqref{eq:edw-Suv},
\[
  \sum_{uv\in E(G)}|S_{uv}|
  =\sum_{uv\in E(G)}
     \bigl(n-d(u)-d(v)+2c_{uv}\bigr).
\]
There are $m$ edges, so the first term gives $mn$.
Every vertex $v$ lies in exactly $d(v)$ edges, so
\[
  \sum_{uv\in E(G)}\bigl(d(u)+d(v)\bigr)
  =\sum_{v\in V(G)}d(v)^2=Q .
\]
Finally $\sum_{uv\in E(G)}c_{uv}=3k_3$, because the
left side counts the pairs $(uv,x)$ with
$x\in N(u)\cap N(v)$, and every triangle gives three
such pairs, one for each of its edges. Therefore
\begin{equation}\label{eq:edw-sum3}
  \sum_{e\in E(G)}|S_e|=mn-Q+6k_3 .
\end{equation}
Combining \eqref{eq:edw-sum1}, \eqref{eq:edw-sum2}
and \eqref{eq:edw-sum3} gives
$ n\,t(G) \le \bk(G)\,\bigl(6t(G)-Q+mn\bigr)$, as needed.
\end{proof}

\end{document}